\documentclass[a4paper, 12pt]{amsart}

\usepackage[a4paper]{geometry}
\usepackage{amssymb}
\usepackage{amsthm}
\usepackage{amsmath}
\usepackage{graphicx}
\usepackage{hyperref}
\usepackage{multirow}
\usepackage{array}
\usepackage[colorinlistoftodos]{todonotes}
\usepackage[english]{babel}
\usepackage{lmodern}
\usepackage{comment}
\usepackage{bbold}
\usepackage{mathtools}
\usepackage{enumerate}
\usepackage{cite}
\usepackage[pagewise,displaymath, mathlines]{lineno}
\usepackage{mathrsfs}

\allowdisplaybreaks

\usepackage[utf8]{inputenc}
\usepackage[T1]{fontenc}

\newtheorem{theorem}{Theorem}[section]
\newtheorem{lemma}[theorem]{Lemma}
\newtheorem{corollary}[theorem]{Corollary}
\newtheorem*{conjecture}{Conjecture}

\theoremstyle{definition}
\newtheorem{definition}[theorem]{Definition}
\newtheorem{proposition}[theorem]{Proposition}

\theoremstyle{remark}
\newtheorem{remark}[theorem]{Remark}
\newcommand{\longcomment}[1]{}

\DeclareMathOperator{\charac}{char}

\DeclareMathOperator{\Frob}{Frob}

\DeclareMathOperator{\rank}{rank}

\DeclareMathOperator{\Hom}{Hom}

\DeclareMathOperator{\Sym}{Sym}

\DeclareMathOperator{\NS}{NS}

\DeclareMathOperator{\End}{End}

\DeclareMathOperator{\Gal}{Gal}
\DeclareMathOperator{\GalQ}{Gal(\overline{\mathbb{Q}}/\mathbb{Q})}
\DeclareMathOperator{\GL}{GL}

\DeclareMathOperator{\Kum}{Kum}
\DeclareMathOperator{\Spec}{Spec}
\DeclareMathOperator{\disc}{disc}
\DeclareMathOperator{\rad}{rad}

\newcommand{\whichbold}[1]{\mathbb{#1}} 

\newcommand{\F}{\whichbold{F}}

\newcommand{\ZZ}{\whichbold{Z}}

\renewcommand{\AA}{\whichbold{A}}
\newcommand{\QQ}{\whichbold{Q}}
\newcommand{\Fq}{\whichbold{F}_{q}}
\newcommand{\Fp}{\whichbold{F}_{p}}



\numberwithin{equation}{section}
\author{Matija Kazalicki}
\address{Department of Mathematics\\ 
	University of Zagreb\\
	Bijeni\v{c}ka cesta 30\\
	10000 Zagreb\\
	Croatia}
\email{matija.kazalicki@math.hr}

\author{Bartosz Naskręcki}
\address{Faculty of Mathematics and Computer Science\\
	Adam Mickiewicz University in Poznań\\
	ul. Uniwersytetu Poznańskiego 4 \\
	61-614, Poznań\\ 
	Poland}
\email{bartnas@amu.edu.pl}

\title[Diophantine triples and K3 surfaces]{Diophantine triples and K3 surfaces}

\makeatletter
\renewcommand{\tocsection}[3]{%
	\indentlabel{\@ifnotempty{#2}{\bfseries\ignorespaces#1 #2\quad}}\bfseries#3}
\renewcommand{\tocsubsection}[3]{%
	\indentlabel{\@ifnotempty{#2}{\ignorespaces#1 #2\quad}}#3}

\newcommand\@dotsep{4.5}
\def\@tocline#1#2#3#4#5#6#7{\relax
	\ifnum #1>\c@tocdepth 
	\else
	\par \addpenalty\@secpenalty\addvspace{#2}%
	\begingroup \hyphenpenalty\@M
	\@ifempty{#4}{%
		\@tempdima\csname r@tocindent\number#1\endcsname\relax
	}{%
		\@tempdima#4\relax
	}%
	\parindent\z@ \leftskip#3\relax \advance\leftskip\@tempdima\relax
	\rightskip\@pnumwidth plus1em \parfillskip-\@pnumwidth
	#5\leavevmode\hskip-\@tempdima{#6}\nobreak
	\leaders\hbox{$\m@th\mkern \@dotsep mu\hbox{.}\mkern \@dotsep mu$}\hfill
	\nobreak
	\hbox to\@pnumwidth{\@tocpagenum{\ifnum#1=1\bfseries\fi#7}}\par
	\nobreak
	\endgroup
	\fi}
\AtBeginDocument{%
	\expandafter\renewcommand\csname r@tocindent0\endcsname{0pt}
}
\def\l@subsection{\@tocline{2}{0pt}{2.5pc}{5pc}{}}
\makeatother
\begin{document}
	

\begin{abstract}
	A Diophantine $m$-tuple with elements in the field $K$ is a set of $m$ non-zero (distinct) elements of $K$ with the property that the product of any two distinct elements is one less than a square in $K$. Let
$X: (x^2-1)(y^2-1)(z^2-1)=k^2,$ 
be an affine variety over $K$. Its $K$-rational points parametrize Diophantine triples over $K$ such that the product of the elements of the triple that corresponds to the point $(x,y,z,k)\in X(K)$ is equal to $k$. We denote by $\overline{X}$ the projective closure of $X$ and for a fixed $k$ by $X_k$ a variety defined by the same equation as $X$.

In this paper, we try to understand what can the geometry of varieties $X_k$, $X$ and $\overline{X}$ tell us about the arithmetic of Diophantine triples.

First, we prove that the variety $\overline{X}$ is birational to $\mathbb{P}^3$ which leads us to a new rational parametrization of the set of Diophantine triples. 

Next, specializing to finite fields, we find a correspondence between a K3 surface $X_k$ for a given $k\in\mathbb{F}_{p}^{\times}$ in the prime field $\mathbb{F}_{p}$ of odd characteristic and an abelian surface which is a product of two elliptic curves $E_k\times E_k$ where $E_k: y^2=x(k^2(1 + k^2)^3 + 2(1 + k^2)^2 x + x^2)$. We derive an explicit formula for $N(p,k)$, the number of Diophantine triples over $\Fp$ with the product of elements equal to $k$. Moreover, we show that the variety $\overline{X}$ admits a fibration by rational elliptic surfaces and from it we derive the formula for the number of points on $\overline{X}$ over an arbitrary finite field $\mathbb{F}_{q}$. Using it we reprove the formula for the number of Diophantine triples over $\Fq$ from \cite{Dujella_Kazalicki_ANT}.

Curiously, from the interplay of the two (K3 and rational) fibrations of $\overline{X}$, we derive the formula for the second moment of the elliptic surface $E_k$ (and thus confirming Steven J. Miller's Bias conjecture in this particular case) which we describe in terms of Fourier coefficients of a rational newform generating $S_4(\Gamma_{0}(8))$.

Finally, in the Appendix, Luka Lasi\'c defines circular Diophantine $m$-tuples, and describes the parametrization of these sets. For $m=3$ this method provides an elegant parametrization of Diophantine triples.	
\end{abstract}
	\maketitle
	\tableofcontents

\section{Introduction}
A Diophantine $m$-tuple with elements in a commutative unital ring $\mathcal{R}$ is a set of $m$
non-zero (distinct) elements of $\mathcal{R}$ with the property that the product of any two of its distinct
elements is one less than a square. 
Diophantus of Alexandria found the first example of a rational Diophantine quadruple
$\{1/16, 33/16, 17/4, 105/16\}$, while the first Diophantine quadruple in integers was found by Fermat, and it was the set $\{1,3,8,120\}$.
Recently, He, Togb\'e and Ziegler \cite{He_Togbe_Ziegler} building upon the work of Dujella \cite{Dujella_Crelle} proved that there does not exist an integer Diophantine quintuple. On the other hand, it was shown in \cite{Dujella_Kazalicki_Mikic_Szikszai} that there are infinitely many rational Diophantine sextuples (for more constructions see \cite{Dujella_Kazalicki}, \cite{Dujella_Kazalicki_Petricevic} and \cite{Dujella_Kazalicki_Petricevic_JNT}), and it is not known if there are rational Diophantine septuples. For a short survey on Diophantine $m$-tuples see \cite{Dujella_Notices}.

In this paper, we focus on Diophantine triples over the field $K$ of odd characteristic that we describe as follows.
Let $$\mathcal{D}:ab+1=r^2,ac+1=s^2,bc+1=t^2$$ be an affine variety in $\mathbb{A}^{6}$, and let $\widetilde{\mathcal{D}}=\mathcal{D}\setminus\{abc(a-b)(a-c)(b-c)=0\}$.
The group $\left(\mathbb{Z}/2\mathbb{Z}\right)^3$ acts on the set of points $\mathcal{D}(K)$ (and also on $\widetilde{\mathcal{D}}(K)$) by multiplying $r,s,t$ coordinates by $\pm 1$,  $(a,b,c,r,s,t) \mapsto (a,b,c,\pm r, \pm s, \pm t)$. It follows from the definition of Diophantine triples that the map $d:\widetilde{D}(K)\rightarrow \AA^3$, $d(a,b,c,r,s,t)=(a,b,c)$ induces a bijection between the set of ordered Diophantine triples over $K$ and the orbits of the $\left(\mathbb{Z}/2\mathbb{Z}\right)^3$-action on the set $\widetilde{D}(K)$.

Moreover, we define 
$$X: (x^2-1)(y^2-1)(z^2-1)=k^2$$ 
to be an affine threefold in $\mathbb{A}^4$, and let $\widetilde{X} = X\setminus\{k(x^2-y^2)(x^2-z^2)(y^2-z^2)=0\}$. 

Our starting point is an observation that the birational map $p:\mathcal{D} \rightarrow X$, $p(a,b,c,r,s,t)=(r,s,t,abc)$ and its inverse $q:X\rightarrow \mathcal{D}$, $q(x,y,z,k) = (\frac{k}{z^2-1}, \frac{k}{y^2-1},\frac{k}{x^2-1},x,y,z)$ define a bijection between the sets $\widetilde{\mathcal{D}}(K)$ and $\widetilde{X}(K)$, hence the following proposition follows.

\begin{proposition}\label{prop:intro}
There is one to one correspondence between ordered Diophantine triples over the field $K$ ($\rm{char}(K)\ne 2$) and $K$-rational points $(\pm r,\pm s,\pm t,k)$ on $X$ with the property that $k(r^2-s^2)(r^2-t^2)(s^2-t^2) \ne 0$.
\end{proposition}

Note that, if we fix $k\in K^\times$, then $K$-rational points on affine variety $X_k: (x^2-1)(y^2-1)(z^2-1)=k^2$ parametrize Diophantine triples with the fixed product $k$.

Our main objective in this paper is to study the geometry of varieties $X$ and $X_k$
 in order to investigate possible applications to the theory of Diophantine triples.

Denote by $\overline{X}$ a projective closure of $X$. We prove in Proposition \ref{prop:param_of_triples} that the projective threefold $\overline{X}$ is birational to $\mathbb{P}^{3}$.
Combining Proposition \ref{prop:param_of_triples} with the correspondence described above, we obtain the following rational parametrization of the Diophantine triples

\begin{align*}
a_1 &= \frac{2 ( {t_1}-1) ( t_1+1)  t_3}{ t_1^2  t_3^2- t_2^2- t_3^2+1} \\
a_2 &=  \frac{2 ( t_2-1) ( t_2+1)  t_3}{ t_1^2  t_3^2- t_2^2- t_3^2+1}, \\
a_3 &= \frac{ t_1^2  t_3^2- t_2^2- t_3^2+1}{2  t_3}.
\end{align*}

In Appendix, Luka Lasi\'c defines and studies circular Diophantine $m$-tuples, and describes another, aesthetically very pleasing symmetric parametrization of Diophantine triples

\begin{align*}
a_1 &=  \frac{2 t_1 (1 + t_1 t_2 (1 + t_2 t_3))} {(-1 + t_1 t_2 t_3) (1 + t_1 t_2 t_3)}, \\
a_2 &=  \frac{2 t_2 (1 + t_2 t_3 (1 + t_3 t_1))} {(-1 + t_1 t_2 t_3) (1 + t_1 t_2 t_3)}, \\
a_3 &=  \frac{2 t_3 (1 + t_3 t_1 (1 + t_1 t_2))} {(-1 + t_1 t_2 t_3) (1 + t_1 t_2 t_3)}.
\end{align*}

Parametric formulas for Diophantine triples can be useful as a starting point in the construction of Diophantine sets and the corresponding elliptic curves of high rank. See Section 2 in \cite{Dujella_Japan} for the first such application to rational Diophantine sextuples, and \cite{Dujella_Kazalicki_Petricevic}, \cite{Dujella_Petricevic} and \cite{Dujella_Peral} for three interesting applications of Luka's formulas.

Next, we switch our attention to Diophantine triples over finite fields.
Diophantine $m$-tuples over the finite fields were first studied in \cite{Dujella_Kazalicki_ANT} where the authors derived a formula for the number of Diophantine triples and quadruples over $\Fp$. 

In Section \ref{sec:fib}, we construct a fibration on the threefold $\overline{X}$ whose fibers are rational elliptic surfaces. For a prime power $q=p^m$, we count $\Fq$-rational points on those fibers, and summing them up derive in Lemma \ref{lemma:pts} the formula for $\#\overline{X}(\Fq)$ 

\begin{equation}\label{eq:i1}
\#\overline{X}(\Fq) = q^3+3q^2+\max\{3 - p, 0\}.
\end{equation}

Using this formula together with Proposition \ref{prop:intro}, in Section \ref{sec:triples} we reprove the formula for the number of Diophantine triples from \cite{Dujella_Kazalicki_ANT}.

For $k\neq 0$, $X_k$ is a singular model of a K3 surface, cf. Lemma \ref{lemma:X_k_is_K3}. In Section \ref{sec:shioda} we show that $X_k$ admits a correspondence with an abelian surface which is a product of two elliptic curves $E_k \times E_k$, where 
$$E_k: y^2=x(k^2(1 + k^2)^3 + 2(1 + k^2)^2 x + x^2)$$
if $k^2\neq -1$. When $k^2=-1$ the curve $E_{k}$ is $y^2=x^3-x$. Using this structure, in Proposition \ref{prop:comparison} and Theorem \ref{thm:Yk} we derive a formula for $\#X_k(\Fp)$, e.g. if $k^2 \ne -1$ then
\begin{equation}\label{eq:i2}
\#X_k(\Fp) = 7-5p+p^2+\left(\frac{k^2+1}{p}\right)(a_{k,p}^2-p),
\end{equation}
where $a_{k,p} = p+1-\#E_k(\Fp)$ and $\left(\frac{\cdot}{p}\right)$ is the Legendre symbol.
From \eqref{eq:i2} and Proposition \ref{prop:intro}, in Corollary \ref{thm:fixed} we derive the formula for $N(p,k)$ - the number of Diophantine triples over $\Fp$ whose product of elements is equal to $k$.

Since surfaces $X_k$ define a fibration of the threefold $X$, we can express $\#X(\Fp)$ as a sum of $\#X_k(\Fp)'s$, hence from \eqref{eq:i1} and \eqref{eq:i2} in Sections \ref{sec:k} and \ref{sec:galois} we obtain the following curious formula

\begin{theorem}
	Let $p$ be an odd prime number. Then
\begin{equation}\label{eq:i3}
\sum_{\substack{k\in \Fp\\ k^2 \notin \{ -1, 0\}}} a_{k,p}^2 = -c_f(p)  
+p^2-3p-2\left(\frac{-1}{p}\right) p-2,
\end{equation}
where
$$f(\tau)=\sum_{n=1}^\infty c_f(n)q^n=q - 4q^3 - 2q^5 + 24q^7 - 11q^9 - 44q^{11} + \ldots \in S_4(\Gamma_0(8)),  q=e^{2\pi i \tau}$$
is a newform which generates the space of cuspforms $S_4(\Gamma_0(8))$ of weight $4$ and level $\Gamma_0(8)$.
\end{theorem}
\begin{remark}
	The newform $f$ is an eta product equal to $\eta(2\tau)^4 \eta(4\tau)^4$, where $\eta(\tau)=q^{\frac{1}{24}} \prod_{n=1}^{\infty} (1-q^{n}) $ is the Dedekind eta function. We note that this modular form is up to factor at $2$ a Mellin transform of the L-series attached to $H^3_{ét}(X,\mathbb{Q}_{\ell})$ cohomology group of certain Calabi-Yau rigid threefold $X$, cf. \cite{Verrill_Rigid_Calabi_Yau}, \cite{Cynk_Meyer}.
\end{remark}

We can rephrase formula \eqref{eq:i3} in the language of Galois representations. In Section \ref{sec:galois}, following Deligne \cite[N$^\circ$ 3]{Deligne_Modular}, to the elliptic surface with generic fiber $E_k$ we associate  a compatible family $(\rho^2_{1,\ell})$ of $\ell$-adic Galois representations of $\GalQ$ with the property that the trace of Frobenius  $Trace(\rho_{1,\ell}^2(\Frob_p))$ is up to some polynomial in $p$ equal to $\displaystyle \sum_{\substack{k\in \Fp\\ k^2 \notin \{ -1, 0\}}} a_{k,p}^2$. 
Now \eqref{eq:i3} is equivalent to the fact that $Trace(\rho_{1,\ell}^2(\Frob_p))$ is equal to the trace of Hecke operator $T_p$ acting on one dimensional space $S_4(\Gamma_0(8))$ generated by $f(\tau)$. This could be readily explained by Deligne's theorem \cite[N$^\circ$ 5]{Deligne_Modular} if $E_k$ were universal elliptic curve for $\Gamma_0(8)$, but this is not the case. 
The elliptic surfaces $\mathcal{E}$ of which the curve $E_{k}$ is a generic fiber is a K3 surface. Its non-equivalent elliptic fibrations have been classified in \cite{Bertin_Lecacheux}. Elliptic fibration $E_{k}$ is described in Section 8.0.9 \cite{Bertin_Lecacheux}. The surface $\mathcal{E}$ has another elliptic fibration which corresponds to the universal family of elliptic curves over $X_{1}(8)$. However, the existence of these two non-equivalent elliptic fibrations on the K3 surface $\mathcal{E}$ is not enough to deduce directly our claims.

We can also rephrase the formula \eqref{eq:i3}  in terms of the second moment sum associated to the elliptic surface  $E_k$. 

In general, let $\mathcal{F}_{k}$ be a $1$-parametric family of elliptic curves over $\QQ$, i.e. a generic fiber of an elliptic surface $\mathcal{F}\rightarrow\mathbb{P}^{1}$ defined over $\QQ$ with a section. For every $k\in\mathbb{P}^1(\mathbb{Q})$ let $\tilde{\mathcal{F}}_{k,p}$ be a minimal model at a rational prime $p$ of $\mathcal{F}_{k}$, and denote by $a_{k,p}$ the number $p+1-\#\tilde{\mathcal{F}}_{k,p}(\mathbb{F}_{p})$.

We define the $r$-th moment sum associated with the $\mathcal{F}_k$ and prime $p$
\[M_{r,p}(\mathcal{F}_k)=\sum_{k\in\mathbb{F}_{p}}a_{k,p}^r,\]
thus it follows from formula \eqref{eq:i3} that for every prime $p\geq 3$ we have
\begin{equation}\label{eq:m2p}
M_{2,p}(E_{k}) = p^2-c_f(p)-\left(3+2\left(\frac{-1}{p}\right)\right)p -1.
\end{equation}
The difference of $1$ between \eqref{eq:m2p} and \eqref{eq:i3} follows from the following calculation. For any prime number $p>2$ the model $E_{k}$ is minimal and singular for $k=0$ (of Kodaira type $I_{4}$). The projective closure of $E_{0}:y^2=x^2(x+2)$ has $p+1+\left(\frac{2}{p}\right)$ points over $\mathbb{F}_{p}$, hence $a_{0,p}^2=1$. If for a given prime $p$ there exists an element $i\in\mathbb{F}_{p}$ such that $i^2=-1$, then at $k=\pm i$ the model $E_{k}$ is minimal and singular of type $III^{*}$. The equation of its reduction is $y^2=x^3$, hence $a_{\pm i,p}=0$.

Rosen and Silverman \cite{Rosen_Silverman} have proved a conjecture of Nagao which implies that the first moment $M_{1,p}(\mathcal{F})$ of a rational elliptic surface is related to the Mordell-Weil rank of the group $\mathcal{F}_{k}(\mathbb{Q}(k))$. In precise terms the limit of $\frac{1}{X}\sum_{p\leq X} M_{1,p}(\mathcal{F}_{k}) \frac{\log p}{p}$ for $X\rightarrow\infty$ exists and is equal to $-\mathrm{rank}\ \mathcal{F}_{k}(\mathbb{Q}(k))$. It follows that the numbers $M_{1,p}(\mathcal{F}_{k})/p$ (i.e. the coefficients $a_{k,p}$) have negative bias (and the larger the rank of the family, the greater the bias). 

It is known from the work of Michel \cite{Michel} that for non-isotrivial families $\mathcal{F}_{k}$ we have
\[M_{2,p}(\mathcal{F}_{k}) = p^2+O(p^{3/2}),\]
More precisely, $M_{2,p}(\mathcal{F}_{k})  = p^2+f_3(p)+f_2(p)+f_1(p)+f_0(p)$ and $|f_i(p)|\leq Cp^{i/2}$ for all primes $p$ and a constant which depends only on the family $\mathcal{F}_{k}$. Each term $f_{i}(p)$ for $i<4$ is a sum of all Galois conjugates of certain algebraic integers of norm $p^{i/2}$.
More precisely, following \cite{Michel} we deduce that  $M_{2,p}(\mathcal{F}_{k})=p^2+p S_{1}(p)+p S_2(p)+ S_{3}(p)$. A function $S_3(p)=O(1)$ depends only on the places of bad reduction of $\mathcal{F}_{k}$, $S_{2}(p)=-\#\{\Delta(\mathcal{F}_{k})=0\}$ is the number of $\mathbb{F}_{p}$ points of the minimal discriminant of $\mathcal{F}_{k}$. Finally $S_{1}=-\textrm{tr}(\Frob_p\mid H^{1}_{c}(U_{p},\Sym_{2}(\mathcal{G}_{p})))$ where $U_{p}$ is a certain open $\mathbb{F}_{p}$-curve computed from $\mathcal{F}_{k}$ and $\mathcal{G}_{p}$ is a rank $2$ lisse sheaf on $U_{p}$ built from the relative family $\mathcal{F}_{k}$, \cite[\S 3]{Michel}. The sum $S_{1}(p)=S_{1}^{(0)}(p)+S_{1}^{(1/2)}(p)$ is such that $S_{1}^{(0)}(p)$ is a sum of algebraic numbers of norm $1$ and $S_{1}^{(1/2)}(p)$ is a sum of algebraic numbers of norm $1/2$. Hence, $f_{0}(p)=S_{3}(p)$, $f_{1}(p)=0$, $f_{2}(p)=p(S_{2}(p)+S_{1}^{(0)}(p))$ and $f_{3}=pS_{1}^{(1/2)}$. We notice that the latter two factorizations are not in the ring of integers, since the sheaf $\Sym_{2}(\mathcal{G}_{p})$ is obtained in the normalization process involving a Tate twist by $1/2$, cf. \cite[\S 3]{Michel}.

 Let $\pi(x)$ denote the number of primes up to $x$. We are interested in the averages $$\mu_i=\mu_i(f_i(p))=\lim_{x\rightarrow\infty}\frac{1}{\pi(x)}\sum_{p\leq x}\frac{f_i(p)}{p^{i/2}}$$ of sequences $(f_i(p))$.

Inspired by the work of Michel, Steven J. Miller in his thesis \cite{Miller_thesis} and subsequently with his co-authors in \cite{Miller_et_al_biases} initiated a more detailed study of bias for the second moments $M_{2,p}(\mathcal{F}_{k})$. They calculated the second moments for certain pencils of cubic of the form $\mathcal{G}_k: y^2=P(x)k + Q(x)$ where $\deg P(x),\deg Q(x)\leq 3$, and based on their findings proposed the following conjecture.

\begin{conjecture}[Bias Conjecture]
	Let $\mathcal{F}_k$ be a one-parameter family of elliptic curves over $\QQ$. The largest index $i$ of the terms $f_i(p)$ in the second moment expansion of $M_{2,p}(\mathcal{F}_{k})$ for which $\mu_i\neq 0$ satisfies the condition $\mu_i<0$.
\end{conjecture}

For a precise statement of the conjecture see the introduction of \cite{Kazalicki_Naskrecki_bias} where we proved the conjecture for the family $\mathcal{G}_k$ for generic choice of polynomials $P(x)$ and $Q(x)$.

\begin{theorem}\label{thm:bias}
The Bias conjecture is true for the family $E_k$. 
\end{theorem}
\begin{proof}
It follows from \eqref{eq:m2p} that $f_{0}=-1$, $f_1(p)=0$, $f_2(p)=-\left(3+2\left(\frac{-1}{p}\right)\right)p$ and  $f_3(p)=-c_f(p)$.

Sato-Tate conjecture \cite[Theorem B.3]{Bartnet_Geraghty_Harris_II} implies
that the sequence of numbers $\{\theta_{p}=c_{f}(p)/(2p^{3/2})\}_{p}$ is equidistributed in the interval $[-1,1]$ with respect to the Sato-Tate measure $\textrm{ST}(t)=2/\pi \sqrt{1-t^2}dt$, i.e. for any continuous function $f$ on $[-1,1]$ we have
\[\lim_{x\rightarrow\infty}\frac{\sum_{p\leq x} f(\theta_{p})}{\pi(x)} = \int_{-1}^{1}f(t)\textrm{ST}(t).\]
In particular, when $f(t)=t$, the integral on the right vanishes and the average of $\frac{c_f(p)}{2\sqrt{p}}$ is zero, hence $\mu_3=0$.

As a consequence of the effective version of Dirichlet's theorem on arithmetic progression, since $(-1/p)=\pm 1$ if and only if $p\equiv \pm 1 \pmod{4}$, we have that  $\lim_{x\rightarrow\infty}\frac{1}{\pi(x)}\sum_{p\leq x}\frac{\left(\frac{-1}{p}\right)p}{p}=0$, thus the highest order lower term of the second moment with a non-vanishing average is $f_2(p)$ term and its average $\mu_2$ is equal to $-3$, hence the bias conjecture holds.
\end{proof}

\section*{Acknowledgements}
The authors would like to thank Anthony Scholl for several comments and suggestions which helped improve both the quality and the content of the paper. We thank Wadim Zudilin for pointing out a possible connection to modular Calabi-Yau threefolds. We would like to thank Odile Lecacheux for pointing out some inaccuracies in the earlier version of this paper. Finally, we would like to thank the anonymous referee for several helpful suggestions which improved the quality of the paper. The authors were supported by the Croatian Science Foundation under the project no.~IP-2018-01-1313. The first author
acknowledges support from the QuantiXLie Center of Excellence, a project
co-financed by the Croatian Government and European Union through the
European Regional Development Fund - the Competitiveness and Cohesion
Operational Programme (Grant KK.01.1.1.01.0004).
\section{Parametrization of the triples}\label{sec:param_of_triples}
Let $X_{k}: (x^2-1)(y^2-1)(z^2-1)=k^2$ be an affine variety defined over a field $K$ of characteristic $\neq 2$. For $k=0$ this is a union of six planes. For $k\neq 0$ it is a singular model of a K3 surface as we will prove in Lemma \ref{lemma:X_k_is_K3}.

When $\charac K=2$ the variety $X_k$ is not reduced. The reduced scheme $X_{k,red}:(x-1)(y-1)(z-1)=k$ is either a union of three lines for $k=0$ or a rational cubic surface for $k\neq 0$.

Let $X$ denote the total space of the fibration $\pi_{K3}:X\rightarrow\mathbb{A}^{1}$, $\pi_{K3}(x,y,z,k)=k$ where the fibers are varieties $X_{k}$. We have another fibration $\pi_{rat}:X\rightarrow\mathbb{A}^{1}$, $\pi_{rat}(x,y,z,k)=z$, where the fibers are denoted $Y_z$.
Let $\overline{X}$ denote the projective closure of the variety $X$
\[\overline{X}:(x^2-w^2)(y^2-w^2)(z^2-w^2)-k^2\cdot w^4=0.\]
We have natural extensions $\pi_{K3}: \overline{X}\rightarrow \mathbb{P}^{1}$, $\pi_{K3}([x:y:z:k:w])=[k:w]$ and $\pi_{rat}:\overline{X}\rightarrow \mathbb{P}^{1}$, $\pi_{rat}([x:y:z:k:w])=[z:w]$.

\begin{proposition}\label{prop:param_of_triples}
	The projective threefold $\overline{X}$ is birational to $\mathbb{P}^{3}$.
\end{proposition}
\begin{proof}
	Let $\phi:\overline{X}\rightarrow\mathbb{P}^{3}$ denote the following rational map
	\begin{equation}
			\phi([x:y:z:k:w]) = [(x+w)y:(x+w)z:kw:(x+w)w]
	\end{equation}
	
	\noindent
	and let $\psi:\mathbb{P}^{3}\rightarrow\overline{X}$ denote the map
	\begin{equation}
		\begin{split}
			\psi([t_1:t_2:t_3:u])=[\left(t_1^2+t_2^2-t_3^2\right) u^3-t_1^2 t_2^2 u-u^5:\\
			-t_1 u^4+t_1 \left(t_1^2+t_2^2+t_3^2\right) u^2-t_1^3 t_2^2:\\
			-t_2 u^4+t_2 \left(t_1^2+t_2^2+t_3^2\right) u^2-t_1^2 t_2^3:\\
			2 t_3 \left(t_1-u\right) \left(t_1+u\right) \left(u-t_2\right) \left(t_2+u\right):\\
			\left(t_1^2+t_2^2+t_3^2\right) u^3-t_1^2 t_2^2 u-u^5]
		\end{split}
	\end{equation}
	
	We check by a direct computation that $\psi\circ\phi$ and $\phi\circ\psi$ are identity maps, hence both are birational and the proof is complete.
\end{proof}
\begin{remark}
To construct a map $\psi$ we observe that, for fixed $A$, $(x^2-1)A=k^2$ defines a conic which goes through the point $x=-1$, $k=0$. The pencil $k=T_3 (x+1)$ of lines determines a parametrization $x=\frac{A+T_3^2}{A-T_3^2}$, $k=\frac{2 A T_3}{A - T_3^2}$. We set $y=T_1$, $z=T_2$ and $A=(T_1^2-1)(T_2^2-1)$. Map $\psi$ is obtained by homogenizing the variables $T_i=t_i/u$ and the formula for its inverse $\phi$ follows from the equation of the pencil of lines.
\end{remark}

\section{Counting points in $\overline{X}(\Fp)$}\label{sec:counting}

In this section, we calculate the formula for $\#\overline{X}(\Fp)$ in two different ways - by using two fibrations of the threefold $\overline{X}$. In Section \ref{sec:fib}, we consider a fibration by rational elliptic surfaces, while in Section \ref{sec:k} we use a fibration by K3 surfaces $X_k$. Using a certain correspondence between $X_k$ and a split abelian surface we describe $\#X_k(\Fp)$ in terms of the $p$-th Fourier coefficient of a certain modular form.
Later in Section \ref{sec:triples} we use this result to derive a formula for the number of Diophantine triples over $\Fp$.

Let $q=p^m$, for a prime $p>2$ and $m\geq 1$. Let $\phi_{q}$ denote the unique multiplicative character $\phi_{q}:\Fq^{\times}\rightarrow \mathbb{C}^{\times}$ of order $2$. We need the following technical result.

\begin{proposition}[\protect{\cite[Thm. 2.1.2]{Berndt_Evans_Williams}}]\label{prop:quadratic_char_summation}
	Let $q$ be a prime power such that $2\nmid q$. Let $\alpha,\beta,\gamma\in\mathbb{F}_{q}$ be given elements. Let $\Delta=4\alpha\gamma-\beta^2$, then
	\begin{equation*}
		\sum_{t \in \mathbb{F}_{q}}\phi_{q}(\alpha t^2+\beta t+\gamma)=\left\{\begin{array}{cc}
			-\phi_q(\alpha), & \alpha\neq 0, \Delta\neq 0\\
			(q-1)\phi_{q}(\alpha) ,& \alpha\neq 0, \Delta=0\\
			q\phi_{q}(\gamma)& \alpha=0, \Delta=0
		\end{array}\right.
	\end{equation*}
\end{proposition}

For a field $K$ of characteristic not equal to $2$ and two elements $k,z\in K$ such that $k\neq 0$ and $z^2\neq 1$ we have an affine curve 

\[X_{k,z}:(x^2-1)(y^2-1)=\frac{k^2}{z^2-1}.\]
Let $a=\sqrt[4]{1-\frac{k^2}{z^2-1}}$. Under the substitution $x\mapsto x'\cdot a,\ y\mapsto y'\cdot a$, the curve $X_{k,z}$ is isomorphic to $(x')^2+(y')^2 = a^2 (1+(x')^2 (y')^2)$ which is famously known as the Edwards normal form \cite{Edwards_curve}. 
The curve of such form has geometric genus $1$, hence is elliptic. In fact, we have a birational map $\phi_{k,z}:X_{k,z} \dashrightarrow Y_{k,z}$ defined over $K$ to a Weierstrass model \cite[III.1]{Silverman_Arithmetic}
\begin{equation}\label{eq:Ykz_model}
Y_{k,z}: Y^2 = X^3 + (4z^4 + (-2k^2 - 8)z^2 + (2k^2 + 4))X^2 + k^4(z^2-1)^2 X.
\end{equation}
where
\[\phi_{k,z}(x,y) = \left(\frac{k^2 (x+1) \left(z^2-1\right)}{x-1},\frac{2 k^2 (x+1) y \left(z^2-1\right)^2}{1-x}\right),\]
\[\phi_{k,z}^{-1}(X,Y) = \left(\frac{2 X}{k^2(1-z^2)+X}-1,\frac{Y}{2X(1-z^2)}\right)\]
The map $\phi_{k,z}$ induces a morphism of degree $1$ 
$$\phi_{k,z}:X_{k,z}\rightarrow Y_{k,z}\setminus \{P_1,P_2,P_3\},$$
where $P_1=(0,0)$, $P_2=(k^2 (z^2-1),-2 k^2 (z^2-1)^2)$ and $P_3=(k^2 \left(z^2-1\right),2 k^2 \left(z^2-1\right)^2)$.

\subsection{Fibration with rational elliptic surfaces}\label{sec:fib}
In this part, we are going to present proof of the following proposition

\begin{lemma}\label{lemma:pts}
	Let $p$ be a prime and let $q=p^m$ for any $m\geq 1$. It follows that

	\begin{equation}\label{eq:X_bar_point_count}
	\#\overline{X}(\Fq) = q^3+3q^2+\max\{3 - p, 0\}
	\end{equation}
	and
	\begin{equation}\label{eq:1}
	\#(X\setminus X_{k=0})(\Fq) = \left\{\begin{array}{lr}
	q^3-6q^2+12q-9, & p>2\\
	q^3-3q^2+3q-1, & p=2
	\end{array}\right.
	\end{equation}
\end{lemma}
\begin{proof}
Let $\overline{X}$ denote a projective closure of the threefold $X$. We have
\[\overline{X}:(x^2-w^2)(y^2-w^2)(z^2-w^2)-k^2\cdot w^4=0.\]

The intersection of $\overline{X}$ with a hyperplane $w=0$ is a union three planes $P_{x}: x=w=0$, $P_{y}: y=w=0$ and $P_{z}:z=w=0$. The planes meet at three distinct lines, which in turn meet at a unique point $[0:0:0:0:1]$. Hence
\[\#\overline{X}(\Fq)= \# X(\Fq)+3q^2+1.\]

For $p=2$ the scheme $X$ is non-reduced. Its reduced subscheme is $(x-1)(y-1)(z-1)=k$, hence $\#X(\Fq) = q^3$. The subscheme $X_{k=0}$ is again non-reduced and its reduced subscheme~is a union of three planes which intersect at three distinct lines and such that have a common point of intersection. Altogether we have $\# X_{k=0}(\Fq) = 3q^2-3q+1$ and $\#(X\setminus X_{k=0})(\Fq) = q^3-3q^2+3q-1$.

Assume that $p>2$. We are going to prove that $\#X(\Fq)=q^3-1$.
We consider the fibration $\pi_{rat}:\overline{X}\rightarrow\mathbb{P}^{1}$ such that $[x:y:z:w]\mapsto [z:w]$.

Let $z\neq \pm 1$. Let $\tilde{Y}_{z}$ denote a smooth projective relatively minimal model of $Y_{z}$ with a natural fibration $\pi_{z}:\tilde{Y}_{z}\rightarrow\mathbb{P}^{1}$. This is a rational elliptic surface \cite[Chap. 7]{Schutt_Shioda} which is at $k=0$ of split type $I_8$ and at $k=\infty$ of non-split type $I_2$. At $k=\pm \sqrt{z^2-1}$ there is a singular fiber of type $I_1$. 
The~fibers cannot coincide under condition $z\neq \pm 1$.
 The~N\'{e}ron-Severi group of $\tilde{Y}_{z}$ is free of rank $10$ and is spanned entirely by the image of the zero section, general fiber and components of the special fibers which is a consequence of the Shioda-Tate formula \cite{Shioda_MW}. 
 For any prime $\ell\neq p$ it follows that $H^{2}_{ét}((\tilde{Y}_{z})_{\overline{\F}_q},\mathbb{Q}_{\ell})$ is spanned entirely by the images of curves from the N\'{e}ron-Severi group.
 Since all the components are defined over $\Fp$ it follows from the Grothendieck-Lefschetz trace formula \cite[C.4]{Hartshorne} that
$$\#\tilde{Y}_{z}(\Fq) = 1+10q+q^2.$$

Let $\tilde{Y}_{k,z}$ denote a fiber $\pi^{-1}_{z}(k)$. For $k=\infty$ we have a fiber of type $I_2$ which is non-split, hence
\begin{equation}
\# \tilde{Y}_{\infty,z}(\Fq) = 2q + 1-\phi_q(z^2-1).
\end{equation}

For $k=0$ we have a fiber of type $I_{8}$ which is split, hence
\begin{equation}
\# \tilde{Y}_{0,z}(\Fq) = 8q.
\end{equation}

If $\phi_{q}(z^2-1)=1$ there are two fibers above $\pm \alpha$, where $\alpha^2 = z^2-1$ in $\Fq$. Both fibers $\widetilde{Y}_{\pm \alpha,z}$ are isomorphic over $\Fq$ to the nodal projective curve $ZY^2=X^3-X^2Z$, hence
\begin{equation}
\# \tilde{Y}_{\pm\alpha,z}(\Fq) = q+1-\phi_{q}(-1).
\end{equation}

From the existence of the fibration $\pi_{z}$ it follows that
\begin{equation}
\# \tilde{Y}_{z}(\Fq)=
 \# \tilde{Y}_{\infty,z}(\Fq)+\# \tilde{Y}_{0,z}(\Fq) + (1+\phi_{q}(z^2-1))\# \tilde{Y}_{\alpha,z}(\Fq)
 + \sum_{\substack{k\in \Fq\\ k(k^2-z^2+1)\neq 0}}\# \tilde{Y}_{k,z}(\Fq).
\end{equation}

For $k\in \Fq$ such that $k(k^2-z^2+1)\neq 0$ we have the equality $\# X_{k,z}(\Fq) = \# \tilde{Y}_{k,z}(\mathbb{F}_{q}) - 4$ since the map $\phi_{k,z}$ is a morphism of degree $1$ onto $\tilde{Y}_{k,z}\setminus\{P_1,P_2,P_3,\mathcal{O}\}$.

On the other hand, the fibration $\pi_{rat}$ on $X$ provides us with the equality
$$\#X(\Fq) = \#X_{z=1}(\Fq)+ \#X_{z=-1}(\Fq)+\sum_{z\in\Fq,z^2\neq 1} \# X_{z}(\Fq).$$

For $z=\pm 1$ the variety $X_{z}$ is isomorphic to $\mathbb{A}^2$, so $\# X_{z=\pm 1}(\Fq) = q^2$.

For $z\neq \pm 1$ it follows that
\[
\# X_{z}(\Fq) = \# X_{0,z}(\Fq)+ (1+\phi_{q}(z^2-1))\# X_{\alpha,z}(\Fq) 
+  \sum_{\substack{k\in \Fq,\\ k(k^2-z^2+1)\neq 0}}\# X_{k,z}(\Fq).
\]
Variety $X_{0,z}$ is a union of four affine lines with a circular arrangement (type $I_4$), hence $\# X_{0,z}(\Fq) = 4q-4$.

For an $\alpha$ such that $\alpha^2 = z^2-1$ in $\Fq$ we have that $X_{\alpha,z}$ is a genus $0$ curve with equation $(x^2-1)(y^2-1)=1$. It follows that
\[\# X_{\alpha,z}(\Fq) = q-3-\phi_{q}(-1).\]
It is easy to check that $\# X_{k=0}(\Fq) = 6q^2-12q+8$ since the scheme $X_{k=0}$ is a configuration of $6$ planes which intersect in the cubical arrangement.
The formulas above lead us to the conclusion that
\[\# X(\Fq) = 2q^2+\sum_{z^2\neq 1} (q^2+\phi_{q}(z^2-1)) = q^3-1.\]
\end{proof}
\begin{remark}
	It is possible to provide a very short computational proof of the formula \eqref{eq:X_bar_point_count} of Lemma \ref{lemma:pts} using only Proposition \ref{prop:param_of_triples}. We compute (in Magma) the explicit schemes $S_{\phi}$ and $S_{\psi}$ such that $\phi,\psi^{-1}: \overline{X}\setminus S_{\phi}\rightarrow \mathbb{P}^{3}\setminus S_{\psi}$ are isomorphisms of schemes. 
	Scheme $S_{\psi}$ is a union of five projective planes
	\[u=0;\quad t_2=\pm u;\quad t_1=\pm u\]
	and a rational surface $\mathcal{P}$
	\[t_1^2 t_3^2 - t_2^2 u^2 - t_3^2 u^2 + u^4=0.\]
	Scheme $S_{\phi}$ is a union of $8$ projective planes
	\[x-w=k=0;\quad x=w=0;\quad y=w=0;\quad z=w=0;\quad y\pm w=k=0;\quad z\pm w=k=0.\]
	It follows that $|S_{\phi}(\mathbb{F}_{q})| = 8(q^2+q+1)-16(q+1)+13$ and $|S_{\psi}(\mathbb{F}_{q})|=5 (q^2 + p + 1) - 16 (q + 1) + 15 + \underbrace{(q^2 + 4 q + 2)}_{=|\mathcal{P}(\mathbb{F}_{q})|}$ and the formula \eqref{eq:X_bar_point_count} easily follows for every field $\mathbb{F}_{q}$ of odd characteristic.
\end{remark}

\subsection{Correspondence between $X_k$ and an abelian surface} \label{sec:shioda}
Let $Z$ denote a K3 surface defined over the complex numbers. The surface $Z$ admits a correspondence with an abelian surface $A$ if there is a diagram of finite maps
\begin{figure}[htb]
	\begin{tikzpicture}
	
	\node (v1) at (-2.5,1.5) {$A$};
	\node (v3) at (1.5,1.5) {$Z$};
	\node (v2) at (-0.5,-0.5) {$\Kum(A)$};
	\draw[->]  (v1) edge[dashed] (v2);
	\draw [->] (v3) edge[dashed] (v2);
	\end{tikzpicture}
	\caption{Correspondence between $Z$ and $A$.}\label{fig:Shioda_Inose}
\end{figure}
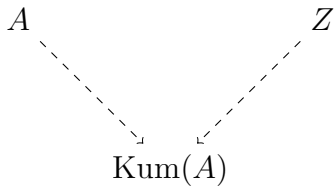

\noindent
where $\Kum(A)$ denotes the Kummer surface attached to $A$, i.e. a resolution of singularities of the quotient surface $A/\langle\pm 1\rangle$. 

\begin{lemma}\label{lemma:X_k_is_K3}
	Let $K$ be a field of characteristic not equal to $2$ and let $k\in K$ be a nonzero element. A surface $X_k$ has a projective smooth minimal model $\widetilde{X}_{k}$ over $K$ which is a K3 surface, i.e. $H^1(\widetilde{X}_{k},\mathcal{O}_{\widetilde{X}_{k}})=0$ and the canonical class $K_{\widetilde{X}_{k}}$ is trivial.
\end{lemma}
\begin{proof}
	Surface $X_k$ is birational to an elliptic surface $\mathcal{Y}_{k}$ with general fiber $Y_{k}: Y^2=f(X,z)$ from \eqref{eq:Ykz_model}.  An application of the Tate algorithm \cite[IV \S 9]{Silverman_Advanced} shows that the Weierstrass model $Y_{k}$ over $K(z)$ is globally minimal in the sense of \cite[\S 5.10]{Schutt_Shioda}. The elliptic discriminant of the model $Y_{k}$ equals $-16 k^8 \left(z^2-1\right)^7 \left(k^2-z^2+1\right)$ which is of degree $16$ with respect to the fibration variable $z$. Hence, the Euler characteristic of the elliptic surface $\mathcal{Y}_{k}$ equals $2$ (cf. \cite[\S 5.13]{Schutt_Shioda}) proving that $\mathcal{Y}_{k}$ is a K3 surface \cite[Chap. 11]{Schutt_Shioda}. It follows from the adjunction formula that for any smooth curve $C$ on a K3 surface $2g(C)-2=C^2$, hence there are no $-1$ curves. This proves that the surface $\mathcal{Y}_{k}$ is minimal, hence we take for $\widetilde{X}_{k}$ the surface $\mathcal{Y}_{k}$.
\end{proof}
\begin{remark}
	Equation $B:f(X,z)=0$ defines a reducible sextic curve on the affine plane. Its projective closure has degree $6$ and all singular points are simple, i.e. at most of multiplicity $3$. Projectivized equation $Y^2=\tilde{f}(X,z,w)$ defines a surface in the weighted projective space $\mathbb{P}(1,1,1,3)$ and determines a normal double cover $T_{k}\rightarrow \mathbb{P}^2$ branched along the sextic curve $B$.
	Its~canonical resolution of singularities $\widetilde{T}_{k}$ has plurigenus $p_g(\widetilde{T}_{k}) = 1$ and irregularity $q(\widetilde{T}_{k})=0$, proving that $\widetilde{T}_{k}$ is a $K3$ surface \cite[V.22]{Barth_Hulek_Peters_Van_de_Ven}.
\end{remark}

In what follows we prove that for each $k\in K$, $k\neq 0$ and $K$ not of characteristic $2$, the elliptic K3 surface $Z=\widetilde{X}_{k}$ admits a correspondence with an abelian surface $A=E\times E$ which is a square of the elliptic curve $E=E_{k}$. In consequence, we obtain an explicit description of the transcendental part of $H^2(Z)$ cohomology (either singular or $\ell$-adic) of $Z$ in terms of the cohomology of $H^1(E)$ cohomology of a given elliptic curve $E$. 

\textbf{Convention:} We use the notation $\mathcal{F}$ to denote an elliptic surface with a generic fiber $F$, which is an elliptic curve over a rational function field.

\begin{lemma}\label{lemma:Shioda_Inose_on_X_k}
	Let $k\in K^{\times}$ be an element of a field $K $ of characteristic not equal to $2$. There exists a finite rational map $f:X_{k}\rightarrow Kum(E_k\times E_k)$ defined over $K(\sqrt{k^2+1})$ where $E_k$ is an elliptic curve over $K$ which depends on $k$. Thus the K3 surface $\widetilde{X}_{k}$ which corresponds to $X_k$ admits a correspondence with the abelian variety $E_{k}\times E_{k}$.
\end{lemma}
\begin{proof}
Let $\psi_{k}:Y_{k}\rightarrow W_{k}$ denote the following $2$-isogeny of two elliptic curves over $K(z)$ with kernel spanned by the point $(0,0)$
\[\psi(x,y):(x,y)\mapsto \left(\frac{k^4 \left(z^2-1\right)^2}{x}-2 k^2 \left(z^2-1\right)+x+4 \left(z^2-1\right)^2,\frac{y \left(x^2-k^4 \left(z^2-1\right)^2\right)}{x^2}\right)\]
where 
\[W_{k}:Y^2=X^3+4 \left(z^2-1\right) \left(k^2-2 z^2+2\right)X^2-16\left(z^2-1\right)^3 \left(k^2-z^2+1\right)X.\]
The surface $\mathcal{W}_{k}$ is isomorphic to a surface $\mathcal{Z}_{k}$ via the isomorphism $(X,Y,z)\mapsto (X,Y,\frac{z+1}{z-1})$ and the generic fiber of the latter has a Weierstrass model
\[Z_{k}:Y^2=X^3+\frac{16 z \left(k^2 z^2-2 k^2 z+k^2-8 z\right)}{(z-1)^4}X^2-\frac{1024 z^3 \left(k^2 z^2-2 k^2 z+k^2-4 z\right)}{(z-1)^8}X.\]
The surface $\mathcal{Z}_{k}$ has singular fibers of the following types ($\sharp$):
\begin{align*}
I_2^*:\  & z=0,\infty,\\
I_4:\  & z=1,\\
I_2:\  & z=\frac{k^2+2\pm 2 \sqrt{k^2+1}}{k^2}.\\
\end{align*}
We identify this elliptic fibration with the fibration $\mathscr{J}_6$ from \cite[\S 2.4]{Kuwata_Shioda}. 

Let us assume for now that $k^2\neq -1$.
Let 
\begin{equation}\label{eq:E_k_curve}
E_k: y^2=x(k^2(1 + k^2)^3 + 2(1 + k^2)^2 x + x^2)
\end{equation}
be an elliptic curve. We have a natural elliptic fibration $(x_1,x_2,t)\mapsto t$ on the Kummer surface $\Kum(E_{k}\times E_{k})$
\[x_1(k^2(1 + k^2)^3 + 2(1 + k^2)^2 x_1 + x_1^2)t^2= x_2(k^2(1 + k^2)^3 + 2(1 + k^2)^2 x_2 + x_2^2).\]
We find an elliptic fibration with reduction types ($\sharp$) when we consider an elliptic parameter $u=\frac{x_1}{x_2}$. Elimination of the variable $x_1$ provides a model
\begin{equation}\label{eq:Kummer}
k^2 \left(k^2+1\right)^3 \left(t^2 u-1\right)+2 \left(k^2+1\right)^2 x_2 (t u-1) (t u+1)+x_2^2 \left(t^2 u^3-1\right)=0.
\end{equation}
Change of variables $(t,x_2)=\left(X,-\frac{\left(k^2+1\right) \left(k^2 u^2 X^2-k^2+u^2 X^2-Y-1\right)}{u^3 X^2-1}\right)$ produces a Weierstrass model
\[Y^2=(1+k^2)(1+ \left(u \left(k^2 (u-1)^2-2 u\right)\right)X^2+u^4 X^4)\]
This is a quadratic twist by $(1+k^2)$ of the quartic singular model of the elliptic curve $Z_{k}$. In~fact, the substitution
\[(X,Y)=\left(-\frac{2 y}{x \left(4 u^2-x\right)},\frac{\left(x-4 u^2\right)^2-2 k^2 (u-1)^2 u \left(2 u^2-x\right)}{x \left(4 u^2-x\right)}\right)\]
produces a model (after twisting by $1+k^2$)
\begin{equation}\label{eq:twisting_model}
y^2=x^3+u \left(k^2 u^2-2 k^2 u+k^2-8 u\right) x^2-4 u^3 \left(k^2 u^2-2 k^2 u+k^2-4 u\right)x.
\end{equation}
A linear change of variables $(x,y,u)=(X\cdot d^2,Y\cdot d^3,z)$ where $d=(z-1)^2/4$ provides an~isomorphism with the model $Z_{k}$.

Let $k^2=-1$. In this case, $Y_{k}$ is $2$-isogenous to a curve which is a generic fiber of the Kummer surface $\Kum(E\times E)$ where $E:y^2=x^3-x$. Calculations are analogous to the general case and so we skip them. A composition of the maps defined above provides the rational map $f$.
\end{proof}

For a given K3 surface $X$ over the complex numbers, we denote by $T_{X}$ the cup-product orthogonal complement in the Betti cohomology $H^{2}(X,\mathbb{Z})$ of the image $N$ of the N\'{e}ron-Severi group via the cocycle map $c_{Betti}:\NS(X)\rightarrow H^{2}(X,\mathbb{Z})$. By the Lefschetz (1,1)-theorem it follows that $N=H^{2}(X,\mathbb{Z})\cap H^{1,1}(X)$. When $X\rightarrow Y$ is a finite map between two K3 surfaces $X,Y$ which induces a non-zero map $H^{2,0}(Y)\rightarrow H^{2,0}(X)$, it follows that $T_{X}$ is Hodge isometric to $T_{Y}$. To check that the former map is non-zero it is enough to verify that a holomorphic two-differential on $Y$ pulls back to a non-zero differential on $X$, cf. \cite[Thm. 12.5]{Schutt_Shioda}. 

For a Kummer surface $\Kum(A)$ attached to an abelian variety $A$, the transcendental lattice $T_{A}\otimes\mathbb{Q}$ is isomorphic to $T_{\Kum(A)}\otimes\mathbb{Q}$ \cite[Prop.4.3]{Morrison_K3}. When $A=E\times E$ is a square of an elliptic curve, the group $T_{A}\otimes\mathbb{Q}$ can be identified with a subgroup of $\Sym^{2}(H^{1}(E,\mathbb{Q}))$ due to K\"{u}nneth formula and the identification $\NS(A)\cong \mathbb{Z}^2\oplus\Hom(E,E)$. When $E$ has no complex multiplication $T_{A}\otimes\mathbb{Q} = \Sym^{2}(H^{1}(E,\mathbb{Q}))$. By the comparison between Betti and $\ell$-adic cohomology we can transfer these results to the etale cohomology. 

The cocycle map induces a $\Gal(\overline{K}/K)$-equivariant injection of the N\'{e}ron-Severi group
\[c:\NS((X)_{\overline{K}})\otimes\mathbb{Q}_{\ell}\hookrightarrow H=H^{2}_{ét}((X)_{\overline{K}},\mathbb{Q}_{\ell}).\]
The comparison theorem identifies $T_{X}\otimes\mathbb{Q}_{\ell}$ with the complement of $im(c)$ in $H$, denoted by~$H^{tr}$.

For the K3 surface $X$ that has a correspondence with $A=E\times E$ over $K$, the group $H^{tr}$ is a Galois submodule in $\Sym^{2}(H^{1}_{ét}(E_{\overline{K}},\mathbb{Q}_{\ell}))$. The two groups are equal when $E$ has no complex multiplication, \cite[\S 12.2.4]{Schutt_Shioda}.

Let $X=\mathcal{Y}_{k}$ and $k\in K\subset\mathbb{C}$, $k\neq 0$.
It follows from the Tate algorithm and the discriminant formula for the generic fiber of the elliptic surface $\mathcal{Y}_{k}$ that we have the following bad fibers. For $k$ such that $k^2\neq -1$ we have
\begin{align*}
	I_1^*:\  & z=\pm 1\\
	I_8:\  & z=\infty\\
	I_1:\  & z=\pm \sqrt{k^2+1}.
\end{align*}
For $k^2=-1$ two fibers of type $I_1$ merge into one fiber above $z=0$ of type $I_2$.

The fibers of type $I_8$ and $I_1^*$ have all components defined over $K$. Let $\ell$ be a prime. The~image of $\NS((\mathcal{Y}_{k})_{\overline{K}})\otimes\mathbb{Q}_{\ell}$ via the cocycle map $c$ is a subspace $N\subset  H^{2}_{ét}((\mathcal{Y}_{k})_{\overline{K}},\mathbb{Q}_{\ell})$.
The subspace $V\subset N$ of the image of $c$ which is spanned by the components of reducible fibers, the image of the zero section and the general fiber has rank equal to at least $19$ due to \cite[Cor. 6.7]{Schutt_Shioda}. Since $\mathcal{Y}_{k}$ is a K3 surface it follows that $\dim_{\mathbb{Q}_{\ell}}H^{2}_{ét}((\mathcal{Y}_{k})_{\overline{K}},\mathbb{Q}_{\ell})=22$.  The subgroup $V$ is a trivial $\Gal(\overline{K}/K)$-representation in $H$.
\begin{proposition}\label{prop:Galois_correspondence}
	Let $K\subset\mathbb{C}$.
	Let $k\in K^{\times}$, $k^2\neq -1$ and $E=E_{k}$ be the elliptic curve \eqref{eq:E_k_curve} with $\End(E)=\mathbb{Z}$. Let $\ell$ be a prime. It follows that the $\rank \NS((\mathcal{Y}_{k})_{\overline{K}})=19$. The~$G$-representation $H$ is equal to $V\oplus T$, where $T= \Sym^2(H^1_{ét}(E)_{\overline{K}},\mathbb{Q}_{\ell})\otimes \chi$ and $\chi$ is a character associated with the extension $K(\sqrt{k^2+1})/K$.
\end{proposition}
\begin{proof}
	Since there is a finite map between $\mathcal{Y}_{k}$ and $\Kum(E_k\times E_k)$ it follows from \cite[Thm. 12.2.5]{Schutt_Shioda} that $\rank \NS((\mathcal{Y}_{k})_{\overline{K}})=\rank \NS((\Kum(E_k\times E_k))_{\overline{K}})$. 
   Next, the formula \cite[Eq. 12.9]{Schutt_Shioda} implies that the latter group has rank $19$, concluding that $T=T_{\mathcal{Y}_{k}}\otimes\mathbb{Q}_{\ell}$ has rank $3$. 
	The discussion before the proposition implies that in that case $T$ is isomorphic as $\Gal(\overline{K}/K(\sqrt{k^2+1}))$-module with $S=\Sym^{2}(H^{1}_{ét}(E_{\overline{K}},\mathbb{Q}_{\ell}))$. By the formula \eqref{eq:twisting_model} $T$ is isomorphic to $S\otimes\chi$ as $\Gal(\overline{K}/K)$-modules for the quadratic character $\chi$ associated with $K(\sqrt{k^2+1})/K$.
\end{proof}
\begin{remark}
	Analogous statements were discussed in  \cite{Geemen_Top} and \cite{Naskrecki_hypergeom}.
\end{remark}

\begin{proposition}\label{prop:sq_m1_case}
	Let $K\subset\mathbb{C}$ and $-1\notin K$. Let $L=K(i)$, where $i^2=-1$. The elliptic surface $\mathcal{Y}_{i}$ is defined over $K$. Its satisfies $\rank\NS((\mathcal{Y}_{i})_{\overline{K}})=20$ and as  $\Gal(\overline{K}/K)$-module $\NS((\mathcal{Y}_{i})_{\overline{K}})$ is a direct sum of rank $19$ trivial representation and a rank $1$-module with basis $<e>$ such that $\sigma(e)=-e$ for the automorphism $\sigma(i)=-i$. The transcendental part of $H^{2}_{ét}((\mathcal{Y}_{k})_{\overline{K}},\mathbb{Q}_{\ell})$ has rank $2$ and is a proper submodule of $\Sym^{2}(H^{1}_{ét}(E_{\overline{K}},\mathbb{Q}_{\ell}))$ where $E:y^2=x^3-x$. 
\end{proposition}
\begin{proof}
	The elliptic surface $\mathcal{Y}_{i}$ has a fiber of type $I_2$ at $z=0$. This fiber is isomorphic over $K$ is a union of two components $F_1,F_2$ defined over $K(i)$, which satisfy the relation $F_1+F_2=F$ in the $\NS((\mathcal{Y}_{i})_{\overline{K}})$, where $F$ is a general fiber. 
	Hence, the representation $\NS((\mathcal{Y}_{i})_{\overline{K}})\otimes\mathbb{Q}$ has a~basis $e_1,\ldots,e_{18}=O, e_{19}=F, e=2F_1-F$. On that basis the action of $\sigma$ is $\sigma(e_i)=e_i$ and $\sigma(e)=-e$, since the elements $e_i$ are defined over $K$ and $\sigma(F_1)=F_2$. 
	
	In this case, the passage from the Kummer model \eqref{eq:Kummer} to the Weierstrass model does not require any twisting, and the argument of Proposition \ref{prop:Galois_correspondence} proves that the transcendental part of $H^{2}_{ét}((\mathcal{Y}_{k})_{\overline{K}},\mathbb{Q}_{\ell})$ is a rank $2$ submodule in $\Sym^{2}(H^{1}_{ét}(E_{\overline{K}},\mathbb{Q}_{\ell}))$.
\end{proof}

\begin{remark}
	We point out that the correspondence between $\widetilde{X}_{k}$ and $E_{k}\times E_{k}$ that we have constructed is not the Shioda-Inose structure \cite{Morrison_K3}. In fact, one can prove that the discriminants of the transcendental lattices of $\widetilde{X}_{k}$ and $\mathcal{W}_{k}$ do not satisfy the necessary condition $\disc T(\widetilde{X}_{k})[2]=\disc T(\mathcal{W}_{k})$.
\end{remark}

\subsection{Fibration with K3 surfaces}\label{sec:k}

In this section we derive formula \eqref{eq:1} (see  Lemma \ref{lemma:pts}) by counting points on fibers $X_k$  and evaluating the sum $$\#(X\backslash X_{k=0})(\Fp)=\sum_{\substack{k \in \Fp\\ k\ne 0}} \#X_{k}(\Fp).$$

We denote by $\tilde{Y}_{k}$ a minimal smooth projective model of $$Y_{k}: Y^2 = X^3 + (4z^4 + (-2k^2 - 8)z^2 + (2k^2 + 4))X^2 + k^4(z^2-1)^2 X,$$ which is an elliptic surface $\pi:\tilde{Y}_{k}\rightarrow \mathbb{P}^{1}$ with projection $\pi$ obtained from the natural projection $(X,Y,z)\mapsto z$.
The following proposition relates $\#X_k(\Fp)$ to $\#\tilde{Y}_k(\Fp)$.

An elliptic surface $N\rightarrow\mathbb{P}^{1}$ which is defined over $\mathbb{Q}$ is said to have a good reduction at~a~rational prime $p$ when there exists a model $\mathcal{N}\rightarrow\Spec\mathbb{Z}_{(p)}=\{g,s\}$ which is smooth projective of relative dimension $2$ and such that the generic fiber $\mathcal{N}_{g}$ is isomorphic with $N$, there exists a~factorization 
\[\mathcal{N}\rightarrow\mathbb{P}^{1}_{\mathbb{Z}_{(p)}}\rightarrow\Spec\mathbb{Z}_{(p)}\]
and the induced morphisms $\mathcal{N}_{x}\rightarrow\mathbb{P}^{1}$ for $x\in\Spec\mathbb{Z}_{(p)}$ define an elliptic surface over $\mathbb{Q}$ and $\mathbb{F}_{p}$, respectively. 

For a polynomial $f\in\mathbb{Z}[t]$, let $\rad(f)$ denote a radical of that polynomial and $\disc(f)$, the~discriminant of $f$.

We say that a Weierstrass model $y^2+a_1 x y+a_3 y = x^3+a_2 x^2+a_4 x+a_6$ of an elliptic curve $N$ over $\mathbb{Q}(t)$ is globally minimal if $a_i\in\mathbb{Z}[t]$, the model is minimal at all finite places and is also minimal at infinity \cite[\S 5.10]{Schutt_Shioda}.
\begin{lemma}\label{lemma:good_reduction}
	 Let $N$ be a globally minimal model of an elliptic curve over $\mathbb{Q}(t)$. Let $\Delta$ be the~discriminant of $N$ and let $j=\frac{f}{g}$ be its j-invariant with $f,g$ coprime and in $\mathbb{Z}[t]$. Let $p>5$ be a rational prime such that for every polynomial $h\in \{\rad(a_i), \rad(\Delta),\rad(f),\rad(g)\}$ we have $p\nmid\disc(h)$ and the modulo $p$ reduction of the equation of $N$ defines an elliptic curve over $\mathbb{F}_{p}(t)$ with Weierstrass coefficients $\widetilde{a}_{i}\in\mathbb{F}_{p}[t]$ defining a globally minimal model over $\mathbb{F}_{p}(t)$, then $p$ is a prime of good reduction for the elliptic surface attached to $N$.
\end{lemma}
\begin{proof}
	This is a translation of a theorem proved by one of the authors in \cite[Tw. 2.2.12]{Naskrecki_thesis}.
	
	The conditions imposed in the theorem imply that two elliptic curves $N/\mathbb{Q}(t)$ and $\tilde{N}/\mathbb{F}_{p}(t)$ satisfy the same steps in Tate's algorithm, hence the associated elliptic curves over $\mathbb{Q}$ and $\mathbb{F}_{p}$, respectively have the same configuration of singular fibers at places which correspond to each other under reduction modulo $p$.
	
	The Weierastrass model $W_{N}\rightarrow\mathbb{P}^{1}_{\mathbb{Z}_{(p)}}$ treated as a relative surface has only isolated double points as singularities \cite[\S 0.2]{Cossec_Dolgachev} in the fibers above both the generic and special fiber.
	In~both characteristics, $0$ and $p$, the types of singularities are the same \cite[Prop. 0.2.6]{Cossec_Dolgachev} and hence the Tate's algorithm provides a simultaneous resolution of singularities on $W_{N}\rightarrow\mathbb{P}^{1}_{\mathbb{Z}_{(p)}}$ leading to a smooth model $\mathcal{N}\rightarrow\Spec\mathbb{Z}_{(p)}$.
\end{proof}

\begin{corollary}\label{cor:good_reduction_cohomology}
	Elliptic curve $Y_{k}$ over $\mathbb{Q}(z)$ has good reduction at every prime $p>5$ and~for~every $k\in\mathbb{Z}$ which satisfies $p\nmid k$. In particular, elliptic surfaces $\mathcal{Y}=\mathcal{Y}_{k}$ over $\mathbb{Q}$ and $\widetilde{\mathcal{Y}}=\widetilde{Y}_{k}$ over $\mathbb{F}_{p}$ satisfy the equality as $D_{p}=\Gal(\overline{\mathbb{Q}}_{p}/\overline{\mathbb{Q}}_{p})$-modules
	\[H^{2}_{ét}((\mathcal{Y})_{\overline{\mathbb{Q}}},\mathbb{Q}_{\ell})=H^{2}_{ét}((\widetilde{\mathcal{Y}})_{\overline{\mathbb{F}}_{p}},\mathbb{Q}_{\ell})\]
	for every prime $\ell\neq p$.
\end{corollary}
\begin{proof}
	The good reduction for primes $p>5$ follows from Lemma \ref{lemma:good_reduction} and a direct calculation with the globally minimal model of $Y_{k}$.  The equality of the cohomology groups as representations of the decomposition group at $p$ follows as a standard fact for $\ell$-adic cohomology of a smooth scheme $\mathcal{Y}\rightarrow\Spec\mathbb{Z}_{(p)}$.
\end{proof}

\begin{proposition} \label{prop:comparison} Let $p$ be an odd prime and $k\in \Fp$. If $k^2 \notin \{-1,0\}$ then $$\#X_k(\Fp)=\#\tilde{Y}_k(\Fp)-24p+6.$$ If $k^2 = -1$ then $\#X_k(\Fp)=\#\tilde{Y}_k(\Fp)-25p+6$.
\end{proposition}
\begin{proof}
	When $z\neq \pm 1$ and $z^2\neq k^2+1$, then $\# X_{k,z}(\Fp) = \#\tilde{Y}_{k,z}-4$ since the fiber $\tilde{Y}_{k,z}$ is an~elliptic curve and the model $Y_{k,z}$ has a point at infinity removed.
	
	When $z=\pm 1$ it follows from the equation of $X_{k,z}$ that $\# X_{k,z}(\Fp) = 0$, while $\#\tilde{Y}_{k,\pm 1}(\Fp) = 6p+1$ because the reduction type at $z=\pm 1$ of the elliptic curve $\tilde{Y}_{k,z}$ is $I_{1}^{*}$ and all the components of the Kodaira fiber are defined over $\Fp$.
	
	At $z=\infty$ we have a split reduction type $I_{8}$ for $\tilde{Y}_{k,z}$, hence $\#\tilde{Y}_{k,z}(\Fp)=8p$. 
	
	For each $z$ such that $z^2=k^2+1\neq 0$ we have a reduction type $I_1$, hence $\#\tilde{Y}_{k,z}(\Fp)=p+1-\phi_{p}(-1)$. On the other hand  for such $z$ the curve $X_{k,z}$ is $x^2 y^2-x^2-y^2= 0$ and $\# X_{k,z}(\Fp)=p-3-\phi_{p}(-1)$.
	
	When $k^2=-1$, then at $z=0$ we have a non-split reduction of type $I_2$ and $\# \tilde{Y}_{k,0}(\Fp) =2p+1-(-1/p)$ while $\# X_{k,0}(\Fp) = p-3-\phi_{p}(-1)$.
	\end{proof}

 Using the correspondence from Section \ref{sec:shioda} attached to $X_k$ (or on $Y_k$ directly) we can derive explicit formula for~$\#\tilde{Y}_k(\Fp)$.  
   
\begin{remark}\label{rem:trace_formula_special_E}
	Let $p$ be an odd prime. Let $\lambda(p)$ denote $p+1-\#E(\mathbb{F}_{p})$ where $E: y^2=x^3-x$.
	Notice that 
 $$\lambda(p)^2=\left\{\begin{array}{ll} 4b^2, & \textrm{if }p \equiv 1 \pmod{4}; p=a^2+b^2 \textrm{ with } b \textrm{ odd},\\
		0, & \textrm{if }p \equiv 3 \pmod{4},
	\end{array}\right.$$
	which follows from \cite[Chap.18, Thm. 5]{Ireland_Rosen}.
\end{remark}

\begin{theorem} \label{thm:Yk} Let $p$ be an odd prime, and $k\in \Fp$. If $k^2 \notin \{-1,0\}$ then
$$\#\tilde{Y}_k(\Fp) = 1+19p+p^2+\phi_p(k^2+1)(a_{k,p}^2-p),$$
where $E_k: y^2=x(k^2(1 + k^2)^3 + 2(1 + k^2)^2 x + x^2)$ and $a_{k,p}= p+1-\#E_{k}(\Fp)$. Moreover, if~$k^2 = -1$ then
$$\#\tilde{Y}_k(\Fp) = 1+(20-\phi_p(-1))p+p^2+(\lambda(p)^2-p).$$
\end{theorem}
\begin{proof}
	For primes $p=3$ and $p=5$ we verify the claims by a direct calculation. Suppose that $p>5$ and $k_{0}$ be a lift of $k\in\mathbb{F}_{p}$, $k^2\neq -1$ to integers which satisfies $k_0>p$. Under such assumption the elliptic curve $E_{k_{0}}$ cannot have complex multiplication since the list of CM $j$-invariants is finite and explicit \cite[App. A \S 3]{Silverman_Advanced}. 
	The conclusion follows from Proposition \ref{prop:Galois_correspondence} and the trace formula applied to Corollary \ref{cor:good_reduction_cohomology}.
	
	Let $k^2=-1$ in $\mathbb{F}_{p}$. Let $\mathcal{Y}_{i}$ be an elliptic surface over $\mathbb{Q}$ as defined in Proposition \ref{prop:sq_m1_case} and let $E_{i}$ denote the elliptic curve $y^2=x^3-x$. Our claim follows from Proposition \ref{prop:sq_m1_case}, Corollary \ref{cor:good_reduction_cohomology} and  Remark \ref{rem:trace_formula_special_E}.
\end{proof}

Proposition \ref{prop:comparison} and Theorem \ref{thm:Yk} now reduce a derivation of the formula \eqref{eq:1} to the~evaluation of the sum $$\displaystyle\sum_{\substack{k \in \Fp,\\ k^2\notin \{-1,0\}}}\phi_p(k^2+1)a_{k,p}^2=-\sum_{\substack{k \in \Fp,\\ k^2\notin \{-1,0\}}}a_{k,p}^2+2\sum_{\substack{k \in \Fp,\\ k^2\notin \{-1,0\},\\ k^2+1=\square}}a_{k,p}^2,$$
where $x=\square$ means that $x$ is a square in $\mathbb{F}_{p}$.   The main idea is to interpret two sums on~the~right-hand side as traces of Frobenii $\Frob_p\in \GalQ$ of certain Galois representations attached to the elliptic surfaces with generic fibers $E_k: y^2=x(k^2(1 + k^2)^3 + 2(1 + k^2)^2 x + x^2)$ and $F_k: y^2 = (x-(k-1)^2)(x-(k+1)^2)x$. 

To motivate the curve $F_k$, we start with the following observation. Let $p$ be an odd prime and for $k \in \Fp \backslash \{-1,0,1\}$ let $b_{k,p}=p+1-\#F_k(\Fp)$.

\begin{proposition}\label{lem:1}
	If $p$ is an odd prime, then
	$$2\lambda(p)^2+2\sum_{\substack{k \in \Fp\\ k^2\notin \{-1,0\}\\ k^2+1=\square}}a_{k,p}^2=\sum_{\substack{k\in \Fp\\ k \notin \{-1,0,1\}}} b_{k,p}^2.$$
\end{proposition}
\begin{proof}
	Let $H_{k}$ denote the elliptic curve obtained from $E_k$ by the pullback through a quadratic map $\phi:k\mapsto \frac{1-k^2}{2k}$. The curve $H_k$ is a quadratic twist by $-(k^2+1)$ of the curve $F_{k}$. Indeed, we have the following models of those curves over $\mathbb{Q}(k)$
	\begin{align*}
	E_{k}:\  y^2& =x^3+2x^2+\frac{k^2}{k^2+1}x,\\
	H_{k}:\ y^2&= x^3+2x^2+\frac{\left(k^2-1\right)^2}{\left(k^2+1\right)^2}x,\\
	F_{k}:\  y^2& =x^3-2(k^2+1)x^2+ \left(k^2-1\right)^2x.
	\end{align*}
	Let $S=\{k\in\mathbb{F}_{p}: k(k^2-1)(k^2+1)\neq 0\}$. Notice that we have an involution $\iota:S\rightarrow S$ with $\iota(k)=-\frac{1}{k}$ without a fixed point on $S$. Hence, there exists two sets $B$, $\tilde{B}$ such that $B\cap \tilde{B} = \emptyset$ and $S=B\cup \tilde{B}$. It follows that $\phi(k)=\phi(\iota(k))$, hence $\phi(B)=\phi(\tilde{B})$. 
	
	Claim: $\phi(B) = \{k\in\mathbb{F}_{p}:k^2\notin\{-1,0\}, k^2+1=\square\}$. Indeed, for an element $s$ we have $\phi(s)^2\in \{-1,0\}$ if and only if $s\in\{\pm 1, \pm \sqrt{-1}\}$ but $\{\pm 1, \pm \sqrt{-1}\}\cap B=\emptyset$. If $k^2+1=r^2$ for some non-zero $r$, then there exists an element $s\in B$ such that $k=\frac{1-s^2}{2s}=\phi(s)$ and $r=\frac{1+s^2}{2s}$.
	
	Let 
	\[\mathcal{S} =\sum_{\substack{k\in \Fp\\ k \notin \{-1,0,1\}}} b_{k,p}^2= \sum_{k \in \{\pm \sqrt{-1}\}}b_{k,p}^2+\sum_{k\in S} b_{k,p}^2\]
	
	We have that $F_{\pm\sqrt{-1}}: y^2=x^3+4x$, so $b_{\pm\sqrt{-1},p}^2= \lambda(p)^2$ since the curves $F_{\pm\sqrt{-1}}$ are quadratic twists of $\tilde{E}_{1728}$. Finally, notice that for $s\in S$ we have the equalities
	\[b_{s,p}^2= a_p(F_s)^2 =a_p(H_s)^2 = a_p(E_{\phi(s)})^2=a_{\phi(s),p}^2,\]
	\[b_{\iota(s),p}^2=a_{\phi(\iota(s)),p}^2=a_{\phi(s),p}^2 \]
	and so
	\[\mathcal{S}-2\lambda(p)^2=\sum_{s\in B}b_{s,p}^2+\sum_{s\in \tilde{B}}b_{s,p}^2=\sum_{s\in B}b_{s,p}^2+\sum_{s\in B}b_{\iota(s),p}^2=2\sum_{s\in B}a_{\phi(s),p}^2=2\sum_{\substack{k \in \Fp\\ k^2\notin \{-1,0\}\\ k^2+1=\square}}a_{k,p}^2.\]
\end{proof}	
\subsection{Compatible families of $\ell$-adic Galois representations of $\GalQ$}\label{sec:galois}
Let $h_1:\mathcal{E}\rightarrow\mathbb{P}^{1}$ and $h_2:\mathcal{F}\rightarrow\mathbb{P}^{1}$ denote two  elliptic surfaces with generic fibers $E_{k}$ and $F_{k}$, respectively. We~will associate to each elliptic surface a compatible family of $\ell$-adic Galois representations of $\GalQ$ as follows.

Denote by $h_1':\mathcal{E}\rightarrow\mathbb{P}^{1}_1$ and $h_2':\mathcal{F}\rightarrow\mathbb{P}^{1}_2$ the restrictions of elliptic surfaces $h_1$ and $h_2$ to~the~complements $\mathbb{P}^{1}_{i}=\mathbb{P}^{1}\setminus\{t\in\mathbb{P}^{1}: h_{i}^{-1}(t)\textrm{ is singular}\}$.

For $j=1,2$, a prime $\ell$, and a positive integer $m$, we obtain a sheaf
\[
\mathcal{F}_\ell^j=R^1 h_j'{_*{\QQ_{\ell}}}
\]
on $\mathbb{P}^{1}_j$, and also a sheaf $i_*\textrm{Sym}^{m}\mathcal{F}_\ell$ on
$\mathbb{P}^{1}_j$ (here $\QQ_\ell$ is the constant sheaf on the elliptic surface $h_j$, $R^1$ is derived functor and $i:\mathbb{P}^{1}_j\rightarrow \mathbb{P}^{1}$ the inclusion). The action of $\GalQ$ on the $\QQ_\ell$-space
\[
W_{m,\ell}^j = H^1_{ét}(\mathbb{P}^{1}_j\otimes \overline{\QQ}, i_*\textrm{Sym}^{m}\mathcal{F}_\ell^j)
\]
defines an $\ell$-adic representation $\rho_{j,\ell}^m$ which is pure of weight $m+1$.

It follows from Lemma \ref{lemma:good_reduction} and a direct computation with globally minimal models of $\mathcal{E}$ and $\mathcal{F}$ that representations  $\rho_{j,\ell}^m$ are unramified for $p>5$. We denote by $Frob_p\in \GalQ$ a~geometric Frobenius at $p$. Using the following well known result, for $p>5$, we can express the trace of $\rho_{j,\ell}^m(\Frob_p)$ in terms of $\#E_k(\Fp)$ or $\#F_k(\Fp)$.

\begin{theorem}
\label{thm:trace}
Let $p>5$ be a prime. For $j=1,2$ and any positive integer $m$ the following are true:
\begin{itemize}
\item[(1)]  We have that
  \[
  Trace(Frob_p|W_{m,\ell}^j)=-\sum_{t\in \mathbb{P}^{1}_j(\Fp) } Trace(Frob_p|(i_*\textrm{Sym}^{m}\mathcal{F}_\ell^j)_t).
  \]
\item[(2)] If the fiber $E^j_t := h_j^{-1}(t)$ (i.e. $E^j_t$ is either equal to $E_t$ or $F_t$) is smooth, then
  \[
  Trace(Frob_p|(i_*\textrm{Sym}^{1}\mathcal{F}_\ell)_t)=Trace(Frob_p|H^1(E^j_t,
  \QQ_\ell))=p+1-\#E^j_t(\mathbb{F}_p).
  \]
	Furthermore,
	$$
	Trace(Frob_p|(i_*\textrm{Sym}^{2}\mathcal{F}_\ell)_t)=Trace(Frob_p|(i_*\textrm{Sym}^{1}\mathcal{F}_\ell^j)_t)^2-p.
	$$
\item[(3)] If the fiber $E^j_t := h_j^{-1}(t)$ is singular, then
   $Trace(Frob_p|(i_*\textrm{Sym}^{2}\mathcal{F}_\ell)_t)=1$ if the fiber is multiplicative or potentially multiplicative.
\item[(4)] If $p \equiv 1 \pmod{4}$ and $t=\pm \sqrt{-1} \in \Fp$, then $Trace(Frob_p|(i_*\textrm{Sym}^{2}\mathcal{F}_\ell^1)_t)=p$.
\end{itemize}
\end{theorem}

\begin{proof}
Points (1)-(3) follow from \cite[Thm. 11.]{Dujella_Kazalicki_ANT}.
We proceed to the proof of point (4).

In order to calculate the trace of $Frob_p|(i_*\textrm{Sym}^{2}\mathcal{F}_\ell^1)_t$ at $t=\pm\sqrt{-1}\in\mathbb{F}_p$, we follow 3.7 of \cite{Sch2}. 
Let $K=\mathbb{F}_{p}(T)$ be the~function field of rational function in $T$. Let $v$ be the~discrete valuation of $K$ corresponding to $T-t$, and $K_v$ the completion. 
Let $G_v$ be the absolute Galois group $\textrm{Gal}(K_v^\textrm{sep}/K_v)$, $I_v$ the inertia group, and $F_v$ a geometric Frobenius. 
Let $H_v$ denote $H^1(E^1\otimes K_v^\textrm{sep}, \QQ_\ell)$. The space $H_v$ is a $G_v$-module and $$Trace(Frob_p|(i_*\textrm{Sym}^2\mathcal{F}_\ell^2)_t ) = Trace(F_v| (Sym^2 H_v)^{I_v}).$$  

First we note that elliptic surface $\mathcal{E}$ has a reduction of type $III^*$ at fibers $t^2=-1$ which is potentially good. We use the fact that for primes $p>5$ the elliptic surface in characteristic $p$ has the same reduction types as the one in characteristic $0$, cf. Lemma \ref{lemma:good_reduction}. So over the local field $K=\mathbb{F}_{p}((v=t-\alpha))$ for $\alpha^2=-1$, $\alpha\in\mathbb{F}_{p}$ we can check that the base change of $E$ to the~field $L=K(\sqrt{\sqrt{v}-v})$ has good reduction. It follows from \cite[Prop. 4.1]{Silverman_Arithmetic} that the inertia group $I_{\overline{L}/L}$ acts trivially on the Tate module $T_{\ell}(E)$. 
The Galois group $\Gal(L/K)$ is cyclic of order $4$. One can check that the inertia group $I_{L/K}$ has order $4$ since the inertia order equals the~ramification degree above the point $[0:1]$ (totally ramified) in the map $f:C\rightarrow \mathbb{P}^{1}$, where $C:Y^2=XZ-X^2$ is a conic and the map $f([X:Y:Z])=[X^2:Z^2]$. 

It follows that the image of inertia subgroup $I_v$ in $\GL(H_v)$ is a cyclic group of order $4$. Since $\Lambda^2 H_v = \QQ_\ell(-1)$ is unramified, there is a basis $\{X,Y\}$ of $H_v\otimes \overline{\QQ_\ell}$ on which $I_v$ acts as $\pm i$. Hence, $(Sym^2 H_v)^{I_v}$ is one-dimensional vector space spanned by $XY$.

Since $p \equiv 1 \pmod{4}$, we have that $\sqrt{-1} \in \Fp$, hence the image of $G_v$ in $\GL(H_v)$ is commutative. In particular, $F_v(X)=\beta X$ and $F_v(Y)=\overline{\beta}Y$, where $\beta$ is an algebraic integer with $\beta\overline{\beta}=p$. It follows that $F_v(XY)=pXY$, and the claim follows.
\end{proof}
\begin{corollary} \label{cor:1}
For a prime $p>5$, we have that
\begin{itemize}
\item[a)]
\[
Trace(\rho_{1,\ell}^2(\Frob_p))  =
 p^2-3p-2\phi_p(-1)p-2-\sum_{\substack{k\in \Fp\\ k^2 \notin \{ -1, 0\}}} a_{k,p}^2,
\]
\item[b)]
\[
 Trace(\rho_{2,\ell}^2(\Frob_p))=
 p^2-3p-4-\sum_{\substack{k\in \Fp\\ k \notin \{-1,0,1\}}} b_{k,p}^2.
\]
\end{itemize}
\end{corollary}
\begin{proof}

The claim follows directly from the previous theorem and the information about the~reduction types of the singular fibers of elliptic surfaces $\mathcal{E}$ and $\mathcal{F}$. Namely, $\mathcal{E}$ has multiplicative singular fibers at $k=0,\infty$ (of types $I_4$ and $I_2$ respectively) and potentially good singular fibers at $k^2+1=0$ (of types $III^*$), while $\mathcal{F}$ has has multiplicative singular fibers at $k=0,\infty, -1, 1$ (of types $I_2, I_2, I_4$ and $I_4$ respectively).
\end{proof}

Consider the newform
$$f(\tau)=\sum_{n=1}^\infty c_f(n)q^n=q - 4q^3 - 2q^5 + 24q^7 - 11q^9 - 44q^{11} + \ldots \in S_4(\Gamma_0(8)).$$

Now we state the main result of this section.

\begin{theorem} \label{prop:main} Let $p>5$ be a prime, and $ \ell \ne p$. We have that
$$Trace(\rho_{1,\ell}^2(\Frob_p))=Trace(\rho_{2,\ell}^2(\Frob_p))=c_f(p).$$
\end{theorem}
\begin{proof}
	\textbf{Second equality} $Trace(\rho_{2,\ell}^2(\Frob_p))=c_f(p)$: the equality follows from Deligne's formula \cite[N$^\circ$ 5]{Deligne_Modular} relating the trace of Frobenius acting on Galois representation of a universal family of elliptic curves over a modular curve and the trace of Hecke operator $T_p$ acting on corresponding space of cusp forms. In this particular case, the family $F_k$ is a twist by $\sqrt{-1}$ of a universal family \cite[Table 3]{Kubert} $$y^2 + xy + \left(-\frac{1}{16} k^2 + \frac{1}{16}\right)y = x^3 + \left(-\frac{1}{16} k^2 +\frac{1}{16}\right) x^2$$  of elliptic curves with $\mathbb{Z}/2\oplus\mathbb{Z}/4$ torsion subgroup. The corresponding modular curve is $X(\Gamma)$ for $\Gamma=\Gamma_{1}(4)\cap \Gamma^{0}(2)$, a congruence subgroup of index $12$, with cusps $\{ \infty,0,1/3,1/2\}$. The following equalities $$\begin{bmatrix*}\frac{1}{\sqrt{2}} & 0\\0& \sqrt{2}
	\end{bmatrix*}^{-1}\Gamma_0(8) \begin{bmatrix*}\frac{1}{\sqrt{2}} & 0\\0& \sqrt{2}
	\end{bmatrix*} = \Gamma_{0}(4)\cap \Gamma^{0}(2)=\{\pm I\} \Gamma,$$  imply that the map $S_4(\Gamma_0(8))\rightarrow S_4(\Gamma)$, $g(\tau)\mapsto g(\tau/2)$ is an isomorphism. Hence, the space $S_4(\Gamma)$ is 1-dimensional and spanned by $f(\tau/2)$.
	
	\textbf{First equality}
	$Trace(\rho_{1,\ell}^2(\Frob_p))=Trace(\rho_{2,\ell}^2(\Frob_p))$:
	
	Corollary \ref{cor:1} shows that our claim is equivalent to:
	\[\sum_{\substack{k\in \Fp\\ k \notin \{-1,0,1\}}} b_{k,p}^2-2\phi_p(-1)p+2= \sum_{\substack{k\in \Fp\\ k^2 \notin \{ -1, 0\}}} a_{k,p}^2.
	\]
	Due to Proposition \ref{lem:1}, we need to prove:
	\[\sum_{\substack{k \in \Fp\\ k^2\notin \{-1,0\}}}\phi_p(k^2+1)a_{k,p}^2=-2 -2\lambda(p)^2+2\phi_p(-1)p.\quad (*)\]
		
     For $p$ an odd prime and $k\in \Fp$ such that $k^2 \notin \{-1,0\}$  Proposition \ref{prop:comparison} implies that $$\#X_k(\Fp)=\#\tilde{Y}_k(\Fp)-24p+6.$$ 
		
	If $k^2 \notin \{-1,0\}$, then it follows from Theorem \ref{thm:Yk} that
		$$\#X_k(\Fp)=7-5p+p^2+\phi_p(k^2+1)(a_{k,p}^2-p).$$
		
		For $k^2=-1$:
		
		$$\#X_k(\Fp)=7-(6+\phi_p(-1))p+p^2+\lambda(p)^2.$$
		
	   It follows from \eqref{eq:1} that
	 $$\sum_{\substack{k \in \Fp\\ k\ne 0}} \#X_{k}(\Fp) = p^3-6p^2+12p-9.$$
	 
	 The following sum
	 $$\sum_{\substack{k \in \Fp\\ k^2\notin \{-1,0\}}} (7-5p+p^2+\phi_p(k^2+1)(a_{k,p}^2-p))+ (1+\phi_p(-1))(7-(6+\phi_p(-1))p+p^2+\lambda(p)^2) $$
	 equals $p^3-6p^2+12p-9$ which is equivalent to
	    
	     $$\sum_{\substack{k \in \Fp\\ k^2\notin \{-1,0\}}}\phi_p(k^2+1)a_{k,p}^2 =  -2 - \lambda(p)^2(1+\phi_p(-1))  + 2 \phi_p(-1) p \quad (**)$$
	     
	     Equality $(**)$ is equivalent to $(*)$ since $-2\lambda(p)^2= - \lambda(p)^2(1+\phi_p(-1))$ is equivalent to
	     $$\lambda(p)^2(\phi_p(-1)-1)=0$$
	     which holds for every odd prime $p$. Indeed, if $p\equiv 1\pmod{4}$, we have $\phi_p(-1)=1$ and for $p\equiv -1\pmod{4}$ we have $\lambda(p)=0$ by Remark \ref{rem:trace_formula_special_E}.
\end{proof}

\begin{theorem} \label{thm:ap} Let $p$ be an odd prime. Then
\[
\sum_{\substack{k\in \Fp\ k^2 \notin \{ -1, 0\}}} a_{k,p}^2 =\begin{cases}
 p^2-5p-2-c_f(p),& \textrm{if } p \equiv 1 \pmod{4},\\
 p^2-p-2-c_f(p),&\textrm{if } p \equiv 3 \pmod{4}.
\end{cases}
\]
and 
\[
\sum_{\substack{k\in \Fp\\ k \notin \{-1,0,1\}}} b_{k,p}^2 =
p^2-3p-4-c_f(p).
\]
\end{theorem}
\begin{proof}
If $p>5$ then the claim follows directly from Corollary \ref{cor:1} and Theorem \ref{prop:main}. In other cases, the claim follows by direct computation.
\end{proof}

\begin{remark}
Note that, arguing as in the proof of Theorem \ref{thm:bias}, it follows from the previous theorem that the second moment formula for the family $F_k$ is
\begin{equation}
	M_{2,p}(F_{k}) = p^2-c_f(p)-3p -1.
\end{equation}
So the family satisfies the Bias conjecture. For the family $H_k$ we obtain the formula
\begin{equation}
	M_{2,p}(H_{k}) = p^2-c_f(p)-3p-2\lambda(p)^2 -1.
\end{equation}
In this case we have the additional contribution of $-2\lambda(p)^2$ which contributes to $f_2(p)$ due to Hasse bound. By the Sato-Tate conjecture for the CM elliptic curve $y^2=x^3-x$ the average of $\lambda(p)^2/p$ equals $1$, hence the Bias conjecture holds for $H_k$ as well.
\end{remark}

\section{Diophantine triples over finite fields} \label{sec:triples}

In this section, by using Lemma \ref{lemma:pts} we prove the following theorem which was first proved in \cite{Dujella_Kazalicki_ANT} using different methods.

\begin{theorem} 
Let $p$ be an odd prime and let $q=p^m$ for any $m\geq 1$. The number of Diophantine triples in $\mathbb{F}_q$ is equal to
\[
N(q) =
			\begin{cases}
			\frac{(q-1)(q-3)(q-5)}{48}, \textrm{ if } q \equiv 1 \pmod{4},\\
			\frac{(q-3)(q^2-6q+17)}{48}, \textrm{ if } q \equiv 3 \pmod{4}.\\
			\end{cases}
\]
\end{theorem}
\begin{remark}
If $q=2^m$ then every element in $\Fq$ is a square (since Frobenius map $x\mapsto x^2$ is an~automorphism of $\Fq$), hence every triple in $\Fq$ will have Diophantine property.
\end{remark}

\begin{proof} Assume that $q=p$.
We will start by carefully studying the correspondence between the~points on $(X\setminus X_{k=0})(\Fp)$ and Diophantine triples in $\Fp$. To the point $(x,y,z,k)\in (X\setminus X_{k=0})(\Fp)$, we associate two different triples $\{\frac{k}{x^2-1},\frac{k}{y^2-1},\frac{k}{z^2-1}\}$ and $\{\frac{-k}{x^2-1},\frac{-k}{y^2-1},\frac{-k}{z^2-1}\}$ (they are different since the product of the elements of the first triple is $k$ and of the second is $-k$). These are Diophantine triples if $x^2, y^2$ and $z^2$ are pairwise distinct. 

Denote by $N_1(p)$ and $N_2(p)$ the number of points in $(X\setminus X_{k=0})(\Fp)$ for which $x^2, y^2$ and $z^2$ are pairwise distinct and such that $xyz=0$ and $xyz\ne 0$ respectively. Since the points $(\pm x,\pm y, \pm z, \pm k)$ all give rise to the same Diophantine triples (as well as the points we get by permuting coordinates $x,y$ and $z$), we have that $$N(p)=\frac{2N_1(p)+N_2(p)}{48}.$$

First, we count the points $(x,y,z,k)\in (X\setminus X_{k=0})(\Fp)$, with $x^2=y^2=z^2$.
This condition is equivalent to $x^2-1=k^2\ne 0$. 
It follows from Proposition \ref{prop:quadratic_char_summation} that the number of such point is equal to $$N_3(p)=\begin{cases}4(p-5)+2, \textrm{ if } p \equiv 1 \pmod{4}\\4(p-3), \textrm{ if } p\equiv 3 \pmod{4}.\end{cases}$$

In the similar way, the number of points $(x,y,z,k)\in (X\setminus X_{k=0})(\Fp)$ with $x^2=y^2\ne z^2$ for some permutation of $x,y$ and $z$ is equal to 
$$N_4(p)=\begin{cases}6p^2-45p+99, \textrm{ if } p \equiv 1 \pmod{4}\\6p^2-45p+81, \textrm{ if } p\equiv 3 \pmod{4}.\end{cases}$$

Finally, to calculate $N_1(p)$ we need to count solutions $(x,y,k)\in \Fp^3$ of equation $(x^2-1)(y^2-1)=-k^2$ with $xyk \ne 0$ and $x^2 \ne y^2$. Since the product $(x^2-1)(y^2-1)$ is a square if both factors are either squares or non-squares, from the previous results it follows that
$$N_1(p)=\begin{cases}3(p^2-10p+25), \textrm{ if } p \equiv 1 \pmod{4}\\3(p^2-6p+9),  \textrm{ if } p\equiv 3 \pmod{4}.\end{cases}$$

Since $N_2(p)=\#(X\setminus X_{k=0})(\Fp)-N_1(p)-N_3(p)-N_4(p)$ from Lemma \ref{lemma:pts} we obtain
$$N_{2}(p)=\begin{cases}p^3-15p^2+83p-165, \textrm{ if } p \equiv 1 \pmod{4}\\ (p-7)(p-5)(p-3),\textrm{ if } p \equiv 3 \pmod{4}\end{cases},$$ and the formula for $N(p)$ follows.

The general case $q=p^m$ is proved analogously.
\end{proof}

Following the proof of the previous theorem, for a prime $p>2$, we can derive the number $N(p,k)$ of Diophantine triples in $\Fp$ with the fixed product $k\in \Fp^\times$.
For $k \in \Fp$ such that $(k^2+1)k \ne 0$, let
 $$G_k: y^2=x^3+x^2-\frac{k^2}{4}x$$
 be an elliptic curve over $\Fp$ which is birationally equivalent to the genus one curve $(x^2-1)^2(y^2-1)=k^2$, and let 
$$H_k:  y^2 = x^3 + (2k^2 + 4)x^2 + k^4 x,$$ be an elliptic curve over $\Fp$ birationally equivalent to the genus one curve $(x^2-1)(y^2-1)=-k^2$. Denote by $c_{k,p}=p+1-\#G_k(\Fp)$ and $d_{k,p}=p+1-\#H_k(\Fp)$. Also, denote by
$e_{k,p}$ the number of distinct $\Fp$-rational solutions of the equation $(x^2-1)^2=-k^2$, and by $f_{k,p}$ the number of distinct $\Fp$-rational solutions of the equation $(x^2-1)^3= k^2$.

\begin{corollary}\label{thm:fixed}
Let $p>3$ be a prime, and let $k \in \Fp^\times$. Then
\begin{enumerate}
\item[a)]	 If $k^2 \ne -1$, then 
\begin{align*}
96 \cdot N(p,k) &=
 2 p^2 +\underbrace{2\phi_p(k^2+1)(a_{k,p}^2-p)-16p}_{=O(p)}\\& +\underbrace{12c_{k,p}-6d_{k,p}}_{=O(\sqrt{p})}+\underbrace{50-12e_{k,p}+16 f_{k,p}-6\phi_p(k^2+1)}_{=O(1)}.
\end{align*}
\item[b)] If $k^2 = -1$, then
$$48\cdot N(p,k)=p^2+\underbrace{(4 b^2-10p)}_{=O(p)}+\underbrace{8f_{k,p}+13}_{=O(1)},$$
where $p=a^2+b^2$ with $b$ odd.
\end{enumerate}
	
\end{corollary}

\section*{Appendix - Circular Diophantine $m$-tuples by Luka Lasić} \label{appendix}

In this appendix, we describe a parametrization of circular Diophantine $m$-tuples. For $m=3$ we obtain a parametrization of rational Diophantine triples for which we show that it is equivalent to the one defined in Section \ref{sec:param_of_triples}.

\begin{definition}
A sequence of $m$ rational numbers $(a_1, a_2, \ldots, a_m)$ is called a circular Diophantine $m$-tuple if $a_{i-1} a_i+1$ is a perfect square for all $i\in \{1,2,\ldots, m\}$ (we use ``circular'' notation: $a_k:=a_{k \bmod{m}+1}$ for all $k \in \ZZ$).
\end{definition}

Given a circular Diophantine $m$-tuple with non-zero elements $(a_1, a_2, \ldots, a_m)$, we define rational numbers $t_i = \frac{1\pm \sqrt{1+a_{i-1}a_i}}{a_i}$ for all $i \in \{1,2,\ldots,m\}$ and for any choice of signs. Define rational functions $F_m=F_m(T_1, T_2, \ldots, T_m) $ and $G_m=G_m(T_1, T_2, \ldots, T_m) $ by

\begin{align*}
	F_m=&\frac{2 T_1(1+T_1 T_2(1+T_2 T_3(1+\ldots(1+ T_{m-2}T_{m-1}(1+T_{m-1}T_m))\ldots)))}{(T_1 T_2 \cdots T_m)^2-1},\\
	G_m=&\frac{1+T_1 T_2(2+T_2 T_3(2+T_3 T_4(2+\ldots (2+T_{m-1}T_{m}(2+T_{m}T_1))\ldots)))}{(T_1 T_2 \cdots T_m)^2-1}.\\
\end{align*}
One can check that $a_i=F_m(t_i,t_{i+1}, \ldots, t_{i+m-1})$, if $t_1 t_2\cdots t_m \ne \pm 1$.

\noindent
\medskip
Conversely, since
\[
F_m(T_i, \ldots, T_{i+m-1})F_m(T_{i+1}, \ldots, T_{i+m})+1 = G_m(T_i, \ldots, T_{i+m-1})^2,
\]
for any choice of $T_i \in \QQ$ such that $T_1 T_2 \cdots T_m \ne \pm 1$ it follows that 
\[
\left(F_m(T_1, T_2, \ldots, T_m), F_m(T_2, T_3, \ldots, T_1), \ldots, F_m(T_m, T_1, \ldots, T_{m-1})\right)
\]
is a circular Diophantine $m$-tuple.

In particular, when $m=3$ circular Diophantine triple $(a_1,a_2,a_3)$ is actually a rational Diophantine triple (provided that elements of the triple are distinct and non-zero). Hence, the~parametrization of circular Diophantine triples provides us with the parametrization of rational Diophantine triples

\begin{align*}
a_1&=\frac{2 t_1 \left(\left(t_2 t_3+1\right) t_1 t_2+1\right)}{t_1^2 t_2^2 t_3^2-1},\\
a_2&=\frac{2 t_2 \left(\left(t_1 t_3+1\right) t_2 t_3+1\right)}{t_1^2 t_2^2 t_3^2-1},\\
a_3&=\frac{2 t_3 \left(\left(t_1 t_2+1\right) t_1 t_3+1\right)}{t_1^2 t_2^2 t_3^2-1}.
\end{align*}

Moreover,
\begin{align*}
r=G_{3}(t_1,t_2,t_3)&=\frac{\left(\left(t_1 t_3+2\right) t_2 t_3+2\right) t_1 t_2+1}{t_1^2 t_2^2 t_3^2-1},\\
s=G_{3}(t_2,t_3,t_1)&=\frac{\left(\left(t_1 t_2+2\right) t_1 t_3+2\right) t_2 t_3+1}{t_1^2 t_2^2 t_3^2-1},\\
t=G_{3}(t_3,t_1,t_2)&=\frac{\left(\left(t_2 t_3+2\right) t_1 t_2+2\right) t_1 t_3+1}{t_1^2 t_2^2 t_3^2-1}
\end{align*}

and
\[\Delta=\frac{8 t_1 t_2 t_3 \left(\left(t_1 t_2+1\right) t_1 t_3+1\right) \left(\left(t_1 t_3+1\right) t_2 t_3+1\right) \left(\left(t_2 t_3+1\right) t_1 t_2+1\right)}{\left(t_1^2 t_2^2 t_3^2-1\right){}^3}.\]

It follows that
\[(r^2-1)(s^2-1)(t^2-1)=\Delta^2,\]
hence we can define a birational map $\mathcal{L}:\mathbb{A}^{3}\rightarrow X$
	\[\mathcal{L}(t_1,t_2,t_3) = (G_{3}(t_1,t_2,t_3),G_3(t_2,t_3,t_1),G_3(t_3,t_1,t_2),\Delta(t_1,t_2,t_3)).\]
The parametrization $\mathcal{L}$ is equivalent to the one we have defined in Section \ref{sec:param_of_triples}.
\begin{corollary}
	The map $\mathcal{L}$ is equal up to a birational automorphism $\mu$ of $\mathbb{A}^{3}$ to the affine restriction $\psi|_{\mathbb{A}^{3}}$ to the first chart of the map $\psi$ from Proposition \ref{prop:param_of_triples}, i.e.
	\[\mathcal{L} =\psi|_{\mathbb{A}^{3}}\circ\mu \]
	where
	\[\mu(t_1,t_2,t_3)=(s,t,a_1 a_2/t_2).\]
\end{corollary}

\bibliographystyle{alpha}
\bibliography{bibliography}

\newcommand{\etalchar}[1]{$^{#1}$}
\begin{thebibliography}{BHPVdV04}

\bibitem[ACF{\etalchar{+}}18]{Miller_et_al_biases}
Megumi Asada, Ryan~C. Chen, Eva Fourakis, Yujin~Hong Kim, Andrew Kwon,
  Jared~Duker Lichtman, Blake Mackall, Steven~J. Miller, Eric Winsor, Karl
  Winsor, Jianing Yang, and Kevin Yang.
\newblock Lower-order biases in the second moment of {D}irichlet coefficients
  in families of {$L$}-functions (with appendices by {S}teven {J}. {M}iller and
  {J}iefei {W}u, and {S}teven {J}. {M}iller and {Y}an {W}eng).
\newblock arXiv:1808.06056, 2018.

\bibitem[BEW98]{Berndt_Evans_Williams}
Bruce~C. Berndt, Ronald~J. Evans, and Kenneth~S. Williams.
\newblock {\em Gauss and {J}acobi sums}.
\newblock Canadian Mathematical Society Series of Monographs and Advanced
  Texts. John Wiley \& Sons, Inc., New York, 1998.
\newblock A Wiley-Interscience Publication.

\bibitem[BHPVdV04]{Barth_Hulek_Peters_Van_de_Ven}
Wolf~P. Barth, Klaus Hulek, Chris A.~M. Peters, and Antonius Van~de Ven.
\newblock {\em Compact complex surfaces}, volume~4 of {\em Ergebnisse der
  Mathematik und ihrer Grenzgebiete. 3. Folge. A Series of Modern Surveys in
  Mathematics [Results in Mathematics and Related Areas. 3rd Series. A Series
  of Modern Surveys in Mathematics]}.
\newblock Springer-Verlag, Berlin, second edition, 2004.

\bibitem[BL13]{Bertin_Lecacheux}
Marie-Jos\'{e} Bertin and Odile Lecacheux.
\newblock Elliptic fibrations on the modular surface associated to
  {$\Gamma_1(8)$}.
\newblock In {\em Arithmetic and geometry of {K}3 surfaces and {C}alabi-{Y}au
  threefolds}, volume~67 of {\em Fields Inst. Commun.}, pages 153--199.
  Springer, New York, 2013.

\bibitem[BLGHT11]{Bartnet_Geraghty_Harris_II}
Tom Barnet-Lamb, David Geraghty, Michael Harris, and Richard Taylor.
\newblock A family of {C}alabi-{Y}au varieties and potential automorphy {II}.
\newblock {\em Publ. Res. Inst. Math. Sci.}, 47(1):29--98, 2011.

\bibitem[CD89]{Cossec_Dolgachev}
Fran\c{c}ois~R. Cossec and Igor~V. Dolgachev.
\newblock {\em Enriques surfaces. {I}}, volume~76 of {\em Progress in
  Mathematics}.
\newblock Birkh\"{a}user Boston, Inc., Boston, MA, 1989.

\bibitem[CM07]{Cynk_Meyer}
Slawomir Cynk and Christian Meyer.
\newblock Modular {C}alabi-{Y}au threefolds of level eight.
\newblock {\em Internat. J. Math.}, 18(3):331--347, 2007.

\bibitem[Del71]{Deligne_Modular}
Pierre Deligne.
\newblock Formes modulaires et repr\'{e}sentations {$l$}-adiques.
\newblock In {\em S\'{e}minaire {B}ourbaki. {V}ol. 1968/69: {E}xpos\'{e}s
  347--363}, volume 175 of {\em Lecture Notes in Math.}, pages Exp. No. 355,
  139--172. Springer, Berlin, 1971.

\bibitem[DK17]{Dujella_Kazalicki}
Andrej Dujella and Matija Kazalicki.
\newblock More on {D}iophantine sextuples.
\newblock In {\em Number theory---{D}iophantine problems, uniform distribution
  and applications}, pages 227--235. Springer, Cham, 2017.

\bibitem[DK21]{Dujella_Kazalicki_ANT}
Andrej {Dujella} and Matija {Kazalicki}.
\newblock {Diophantine $m$-tuples in finite fields and modular forms}.
\newblock {\em Research in Number Theory}, 7, 3, 2021.

\bibitem[DKMS17]{Dujella_Kazalicki_Mikic_Szikszai}
Andrej Dujella, Matija Kazalicki, Miljen Miki\'{c}, and M\'{a}rton Szikszai.
\newblock There are infinitely many rational {D}iophantine sextuples.
\newblock {\em Int. Math. Res. Not. IMRN}, (2):490--508, 2017.

\bibitem[DKP19]{Dujella_Kazalicki_Petricevic_JNT}
Andrej Dujella, Matija Kazalicki, and Vinko Petri\v{c}evi\'{c}.
\newblock Rational {D}iophantine sextuples with square denominators.
\newblock {\em J. Number Theory}, 205:340--346, 2019.

\bibitem[DKP20]{Dujella_Kazalicki_Petricevic}
Andrej Dujella, Matija Kazalicki, and Vinko Petri\v{c}evi\'{c}.
\newblock Rational {D}iophantine sextuples containing two regular quadruples
  and one regular quintuple.
\newblock {\em Acta Mathematica Spalatensia}, 1:19--27, 2020.

\bibitem[DP20a]{Dujella_Peral}
Andrej Dujella and Juan~Carlos Peral.
\newblock High rank elliptic curves induced by rational {D}iophantine triples.
\newblock {\em Glas. Mat. Ser. III}, 55(4):237--252, 2020.

\bibitem[DP20b]{Dujella_Petricevic}
Andrej Dujella and Vinko Petri\v{c}evi\'{c}.
\newblock Doubly regular {D}iophantine quadruples.
\newblock {\em Rev. R. Acad. Cienc. Exactas F\'{\i}s. Nat. Ser. A Mat. RACSAM},
  114(4):Paper No. 189, 8, 2020.

\bibitem[Duj04]{Dujella_Crelle}
Andrej Dujella.
\newblock There are only finitely many {D}iophantine quintuples.
\newblock {\em J. Reine Angew. Math.}, 566:183--214, 2004.

\bibitem[Duj09]{Dujella_Japan}
Andrej Dujella.
\newblock Rational {D}iophantine sextuples with mixed signs.
\newblock {\em Proc. Japan Acad. Ser. A Math. Sci.}, 85(4):27--30, 2009.

\bibitem[Duj16]{Dujella_Notices}
Andrej Dujella.
\newblock What is \dots a {D}iophantine {$m$}-tuple?
\newblock {\em Notices Amer. Math. Soc.}, 63(7):772--774, 2016.

\bibitem[Edw07]{Edwards_curve}
Harold~M. Edwards.
\newblock A normal form for elliptic curves.
\newblock {\em Bull. Amer. Math. Soc. (N.S.)}, 44(3):393--422, 2007.

\bibitem[Har77]{Hartshorne}
Robin Hartshorne.
\newblock {\em Algebraic geometry}.
\newblock Springer-Verlag, New York-Heidelberg, 1977.
\newblock Graduate Texts in Mathematics, No. 52.

\bibitem[HTZ19]{He_Togbe_Ziegler}
Bo~He, Alain Togb\'{e}, and Volker Ziegler.
\newblock There is no {D}iophantine quintuple.
\newblock {\em Trans. Amer. Math. Soc.}, 371(9):6665--6709, 2019.

\bibitem[IR90]{Ireland_Rosen}
Kenneth Ireland and Michael Rosen.
\newblock {\em A classical introduction to modern number theory}, volume~84 of
  {\em Graduate Texts in Mathematics}.
\newblock Springer-Verlag, New York, second edition, 1990.

\bibitem[KN20]{Kazalicki_Naskrecki_bias}
Matija Kazalicki and Bartosz Naskręcki.
\newblock Second moments and the bias conjecture for the family of cubic
  pencils.
\newblock arXiv:2012.11306, 2020.

\bibitem[KS08]{Kuwata_Shioda}
Masato Kuwata and Tetsuji Shioda.
\newblock Elliptic parameters and defining equations for elliptic fibrations on
  a {K}ummer surface.
\newblock In {\em Algebraic geometry in {E}ast {A}sia---{H}anoi 2005},
  volume~50 of {\em Adv. Stud. Pure Math.}, pages 177--215. Math. Soc. Japan,
  Tokyo, 2008.

\bibitem[Kub76]{Kubert}
Daniel~Sion Kubert.
\newblock Universal bounds on the torsion of elliptic curves.
\newblock {\em Proc. London Math. Soc. (3)}, 33(2):193--237, 1976.

\bibitem[Mic95]{Michel}
Philippe Michel.
\newblock Rang moyen de familles de courbes elliptiques et lois de
  {S}ato-{T}ate.
\newblock {\em Monatsh. Math.}, 120(2):127--136, 1995.

\bibitem[Mil02]{Miller_thesis}
Steven~Joel Miller.
\newblock {\em 1- and 2-level densities for families of elliptic curves:
  {E}vidence for the underlying group symmetries}.
\newblock ProQuest LLC, Ann Arbor, MI, 2002.
\newblock Thesis (Ph.D.)--Princeton University.

\bibitem[Mor84]{Morrison_K3}
David~R. Morrison.
\newblock On {K}3 surfaces with large {P}icard number.
\newblock {\em Invent. Math.}, 75(1):105--121, 1984.

\bibitem[Nas14]{Naskrecki_thesis}
Bartosz Naskręcki.
\newblock Ranks in families of elliptic curves and modular forms.
\newblock {\em Adam Mickiewicz University in Poznań}, 2014.
\newblock Ph.D. thesis.

\bibitem[{Nas}17]{Naskrecki_hypergeom}
Bartosz {Naskr\k{e}cki}.
\newblock On a certain hypergeometric motive of weight 2 and rank 3.
\newblock arXiv:1702.07738, 2017.

\bibitem[RS98]{Rosen_Silverman}
Michael Rosen and Joseph~H. Silverman.
\newblock On the rank of an elliptic surface.
\newblock {\em Invent. Math.}, 133(1):43--67, 1998.

\bibitem[Sch88]{Sch2}
Anthony~J. Scholl.
\newblock The {$l$}-adic representations attached to a certain noncongruence
  subgroup.
\newblock {\em J. Reine Angew. Math.}, 392:1--15, 1988.

\bibitem[Shi90]{Shioda_MW}
Tetsuji Shioda.
\newblock On the {M}ordell-{W}eil lattices.
\newblock {\em Comment. Math. Univ. St. Paul.}, 39(2):211--240, 1990.

\bibitem[Sil94]{Silverman_Advanced}
Joseph~H. Silverman.
\newblock {\em Advanced topics in the arithmetic of elliptic curves}, volume
  151 of {\em Graduate Texts in Mathematics}.
\newblock Springer-Verlag, New York, 1994.

\bibitem[Sil09]{Silverman_Arithmetic}
Joseph~H. Silverman.
\newblock {\em The arithmetic of elliptic curves}, volume 106 of {\em Graduate
  Texts in Mathematics}.
\newblock Springer, Dordrecht, second edition, 2009.

\bibitem[SS19]{Schutt_Shioda}
Matthias Sch\"{u}tt and Tetsuji Shioda.
\newblock {\em Mordell-{W}eil lattices}, volume~70 of {\em Ergebnisse der
  Mathematik und ihrer Grenzgebiete. 3. Folge. A Series of Modern Surveys in
  Mathematics [Results in Mathematics and Related Areas. 3rd Series. A Series
  of Modern Surveys in Mathematics]}.
\newblock Springer, Singapore, 2019.

\bibitem[Ver00]{Verrill_Rigid_Calabi_Yau}
H.~A. Verrill.
\newblock The {$L$}-series of certain rigid {C}alabi-{Y}au threefolds.
\newblock {\em J. Number Theory}, 81(2):310--334, 2000.

\bibitem[vGT06]{Geemen_Top}
Bert van Geemen and Jaap Top.
\newblock An isogeny of {K}3 surfaces.
\newblock {\em Bull. London Math. Soc.}, 38(2):209--223, 2006.

\end{thebibliography}
\end{document}